\documentclass[11pt,reqno]{amsart}
\usepackage{graphicx}
\usepackage{color}
\usepackage{mathrsfs,amsmath,amssymb,amsthm,amsfonts}

\usepackage{graphics,url,color,epsfig}
\usepackage{tikz}
\usepackage[colorlinks]{hyperref}
\AtBeginDocument{%
  \hypersetup{%
    linkcolor=blue,%
    citecolor=blue,%
  }%
}
\usepackage{cleveref}
\usepackage{pgfplots}
\usepackage{multicol}

\makeatletter

\newcommand{\Rmnum}[1]{\expandafter\@slowromancap\romannumeral #1@}
\makeatother

\newtheorem{thm}{Theorem}[section]

\newtheorem{lemma}[thm]{Lemma}
\newtheorem{remark}[thm]{Remark}
\newtheorem{theorem}[thm]{Theorem}

\usepackage[numbers,sort&compress]{natbib}
\newcommand{\lap}{\mbox{$\Delta$}}
\newcommand{\la}{\lambda}
\newcommand{\be}{\begin{equation}}
\newcommand{\ee}{\end{equation}}

\everymath{\displaystyle}

\allowdisplaybreaks

\begin{document}

\author{Pengyan Wang}
\address{Pengyan Wang \newline\indent School of Mathematics and Statistics \newline\indent  Xinyang Normal University \newline\indent Xinyang, 464000, P. R. China}
\email{wangpy@xynu.edu.cn}

\author{Leyun Wu}
\address{Leyun Wu
\newline\indent
School of Mathematics \newline\indent South China University of Technology \newline\indent Guangzhou, 510640, P. R. China}
\email{leyunwu@scut.edu.cn}

\title{Liouville   theorems for  mixed local and  nonlocal  indefinite equations}

\begin{abstract}
We investigate the qualitative properties of positive solutions to mixed local-nonlocal equations with indefinite nonlinearities, emphasizing the interaction between classical and fractional Laplacians. We first establish maximum principles and prove strict monotonicity along the \(x_1\)-direction for mixed elliptic operators. By combining a mollified first eigenfunction with a suitable sub-solution, we derive nonexistence results for the mixed operator \( (-\Delta)^s - \Delta \) via a contradiction argument. These results are further extended to the parabolic setting, incorporating both the Marchaud-type fractional time derivative and the classical first-order derivative, revealing new qualitative features under dual nonlocality.

A key aspect of our approach is a careful adaptation of the method of moving planes to the mixed local--nonlocal context. By addressing the distinct scaling behaviors of local and nonlocal terms, the method yields monotonicity and Liouville-type results without standard decay assumptions, and provides a framework potentially applicable to a broader class of mixed elliptic and parabolic problems.
\end{abstract}





\maketitle

\numberwithin{equation}{section}

\section{Introduction and main results}

It is worth noting that interest in problems involving mixed operators has grown significantly in recent years. These operators arise from the interplay between two stochastic processes operating at different spatial scales: the classical random walk and the L\'{e}vy flight. When a particle follows either of these processes with certain probabilities, the resulting limiting diffusion behavior is governed by an operator of mixed order. A detailed analysis of this phenomenon and its implications can be found in \cite{DPLV1}.
This class of operators offers a flexible framework for exploring the contrasting roles of local and nonlocal diffusion mechanisms in various applied settings. For example, they have been used to model how different levels of regional and global movement restrictions might affect the spread of infectious diseases (see \cite{EGYMWB}).
In the context of free-ranging animals, empirical evidence of shifts between L\'{e}vy and Brownian motion patterns indicates that animals adaptively adjust their movement strategies in response to heterogeneous distributions of environmental resources (see \cite{HQD}). Classical applications of such mixed dynamics also include heat transport in magnetized plasmas, as discussed in \cite{Bd2013}.

The objective of this paper is to explore the monotonicity and the Liouville theorems for mixed local and nonlocal equations with indefinite nonlinearities. 

First, we concern the follwing  mixed local and nonlocal elliptic equations with indefinite nonlinearities:
\begin{equation}\label{v1}
(-\lap)^su(x)-\lap u(x)=f(x_1)g(u(x)),  ~x \in \mathbb R^n,
\end{equation}
 As is well
 known the fractional Laplacian is defined, up to a normalization constant, by a singular integral
\begin{eqnarray*}
\begin{aligned}
(-\Delta)^s u(x)&=C_{n, s}P.V. \int_{\mathbb{R}^n}\frac{u(x)-u(y)}{|x-y|^{n+2s}}dy\\
&=C_{n, s} \lim_{\varepsilon \to 0}\int_{\mathbb{R}^n\backslash B_\varepsilon(x)}\frac{u(x)-u(y)}{|x-y|^{n+2s}}dy,
\end{aligned}
\end{eqnarray*}
where $s\in (0, 1),$ $P.V.$
stands for the Cauchy principal value, $B_{\varepsilon}(x)$ is the ball of radius $\varepsilon$ centered at $x,$
and $C_{n, s}$ is a
dimensional constant depending on $n$ and $s$, precisely given by
$$
C_{n, s}=
=\frac{ 4^s s\Gamma(\frac{n+2s}{2})}{\pi^{\frac{n}{2}}\Gamma(1-s)},
$$
where $\Gamma$ denotes the Gamma function. 
In order to ensure the well-definition of mixed operator, it is necessary to have $u\in C^2 (\mathbb{R}^n) \cap \mathcal{L}_{2s}$ with the space $\mathcal{L}_{2s}$ defined as 
$$
\mathcal{L}_{2s}:=\left\{ u: \mathbb{R}^n \to \mathbb{R} \mid
\int_{\mathbb{R}^n}\frac{|u(x)|}{1+|x|^{n+2s}}dx<+\infty\right\}.
$$

Moreover, we also study the following indefinite  mixed local and  nonlocal parabolic equations involving the nonlocal time derivative:
\begin{eqnarray}\label{eq1}
  {\partial^\alpha_ t}u(x,t) +(-\Delta)^s u(x,t)-\lap u(x,t)= f(x_1) g(u(x, t)),\,\, (x, t) \in  \mathbb{R}^n \times \mathbb{R},
\end{eqnarray}
where  $ {\partial^\alpha_ t}$ is known as the Marchaud fractional derivative of order $\alpha$, defined by
$$ {\partial^\alpha_ t}u(x,t) =C_\alpha \int_{-\infty}^t \frac{u(x,t)-u(x,\tau)}{(t-\tau)^{1+\alpha}}d\tau,$$
where $ \alpha \in(0,1)$ and $C_\alpha =\frac{\alpha}{\Gamma(1-\alpha)}$. The Marchaud fractional derivative captures the irreversibility of the transport process in time and its inherent memory effect, and it reduces to the classical first-order time derivative  $ \partial_t u $  when $\alpha=1$.
In the parabolic setting, $(-\Delta)^s u(x, t)$  is defined in the same way as in the elliptic case. For each fixed $t>0$,
\begin{equation*}\aligned
(-\lap)^s u(x,t)&=C_{n,s} P.V. \int_{\mathbb R^n} \frac{u(x,t)-u(y,t)}{|x-y|^{n+2s}}dy\\
&= C_{n,s}  \lim\limits_{\epsilon\rightarrow 0}\int_{\mathbb R^n\setminus B_{\epsilon}(x)}\frac{u(x,t)-u(y,t)}{|x-y|^{n+2s}}dy.
\endaligned \end{equation*}
To ensure the well-definition of the left-hand side in  \eqref{eq1},
we assume 
$$
u(x,t)\in (C^2(\mathbb R^n)\cap {\tilde {\mathcal L}}_{2s})\times (C^1(\mathbb R)\cap {\mathcal L}^-_\alpha (\mathbb R)),
$$
where  
$$ {\tilde {\mathcal L}}_{2s}=\left\{u(\cdot, t) \in L^1_{loc} (\mathbb{R}^n) \mid \int_{\mathbb R^n} \frac{|u(x,t)|}{1+|x|^{n+2s}}dx<+\infty\right\},$$
and
$${\mathcal L}^-_{\alpha}(\mathbb R):=\left\{ u\in L^1_{loc}(\mathbb R ) \mid \int_{-\infty}^t \frac{|u(\tau)|}{1+|\tau|^{1+\alpha}}d\tau<+\infty~\mbox{for each }~t\in \mathbb R\right\}.$$

Liouville theorems are powerful tools for identifying when partial differential equations (PDEs) admit no nontrivial global solutions. They play a central role in establishing nonexistence, uniqueness, regularity, and classification results, and serve as fundamental instruments in both pure and applied analysis. For instance, they have been crucial in deriving a priori bounds for solutions in \cite{BPGQ, CLL2016, FLN, GS1}, and in proving uniqueness results in \cite{C, KK, MS}. For scaling-invariant superlinear parabolic equations and systems, Liouville theorems ensure optimal universal estimates for solutions to associated initial and initial-boundary value problems; see \cite{PQS, QS} and references therein. 

When considering the indefinite elliptic problem for Laplacian
\begin{eqnarray*}
    \begin{cases}
    -\Delta u=a(x) u^p &\mbox{in}\,\, \Omega,\\
    u=0& \mbox{on}\,\, \partial \Omega,
    \end{cases}
\end{eqnarray*}
where $\Omega$ is a smooth bounded domain in $\mathbb{R}^n$, and $a(x)$ is a sign-changing function  satisfying some certain conditions.  
To derive a priori bounds for solutions, one can perform a blow-up and rescaling procedure. In certain cases, this leads to the following limiting equation:
\begin{eqnarray*}
    -\Delta u=x_1 u^p &\mbox{in}\,\, \mathbb{R}^n.\\
\end{eqnarray*}
To reach a contradiction and thereby establish the a priori bounds, the Liouville theorem for this equation is required,  we refer the readers  to  \cite{BCN,DLi,Lin} and the references therein. For the fractional indefinite elliptic problem
\begin{eqnarray}\label{001}
    (-\Delta)^s u=x_1 u^p &\mbox{in}\,\, \mathbb{R}^n.
\end{eqnarray}
Chen-Zhu \cite{CZ} established a Liouville theorem for 
positive solutions to \eqref{001} in the case  $\frac{1}{2}<s<1$ by the extension method. Subsequently, Barrios et al. \cite{BPGQ2017}  extend this result to $0<s<1$  with more general nonlinearities.
At the same time,  Chen-Li-Zhu \cite{CLZ}  proved the Liouville theorem  of
positive solutions for \eqref{001} employing a direct method of moving planes
under significantly weaker conditions than that in \cite{BPGQ2017}.


  For   parabolic indefinite problems, Pol\'{a}\v{c}ik-Quittner \cite{PQ} studied a Liouville theorem nonnegative bounded solutions of the following parabolic equation
  \begin{equation*}
  \partial_ t u-\lap u(x,t)=a(x_1)f(u),~x=(x_1,x_2,\cdots,x_n)\in \mathbb R^n,~t\in \mathbb R,
  \end{equation*}
  where $a(x_1)$ is an increasing function tending to infinity when $x_1 \to +\infty$.
Chen-Wu-Wang \cite{CWW} investigated indefinite fractional parabolic problems and proved a Liouville theorem for positive solutions to
$$ \partial_t u(x,t)+(-\lap)^su(x,t)=x_1 u^p(x,t),~(x,t)\in \mathbb R^n \times \mathbb R.$$
Subsequently, Huang and Zhong \cite{HZ} established a Liouville theorem for tempered fractional parabolic equations with an indefinite term. Later, Chen and Guo \cite{CG} derived a Liouville theorem for indefinite dual fractional parabolic equations involving the Marchaud fractional time derivative.
For further related work on indefinite equations, see \cite{BCN,CL1997,CL1998,DLi,WYz,Zm} and the references therein.

Recently, the qualitative properties of solutions to problems involving mixed local and nonlocal operators have attracted significant attention. For example, De Filippis and Mingione \cite{DeM2024} studied the local H\"{o}lder continuity of the gradient of minimizers and provided the first boundary regularity results for such solutions. Anthal and Garain \cite{AG2025MZ} investigated symmetry, existence, nonexistence, and regularity of weak solutions to a mixed local and nonlocal singular jumping problem. Ding and Zhang \cite{DFZ2024JLMS} examined the local behavior of minimizers, including local boundedness, local Hölder continuity, and Harnack inequalities. Biagi and Vecchi \cite{BV2024CVPDE} explored the multiplicity of solutions to mixed local and nonlocal semilinear elliptic problems with singular and critical nonlinearities.

Motivated by the aforementioned works, the objective of this paper is to establish monotonicity and  Liouville theorems for positive solutions to mixed local and nonlocal equations  \eqref{v1} and \eqref{eq1}. To the best of our knowledge,  monotonicity and Liouville theorems for such equations have not yet been established in the existing literature. As a result, traditional methods that rely on constructing auxiliary functions to force decay of solutions at infinity are not applicable to indefinite problems involving mixed local and nonlocal operators.
In this paper, we adopt a direct method of moving planes, inspired by \cite{CLL, CLZ}, to address this problem. This method has been successfully applied to a wide range of equations, including the heat equation, the Allen–Cahn equation, and equations involving the fractional Laplacian (see \cite{CH2021JFA, CWNH, CWW, DQ2018AM, DW2024MZ}, among others).

Before stating our main results, 
we define a classical solution to equation \eqref{v1} as a function $u(x)\in C^2(\mathbb R^n)\cap {\mathcal L}_{2s} $ that  satisfies \eqref{v1}
pointwise in  $\mathbb R^n$. Simililarly, 
a function $u(x, t)$ is said to be  classical solution of \eqref{eq1}, if $$u(x,t)\in (C^2(\mathbb R^n)\cap {\tilde{\mathcal L}}_{2s})\times ( C^1(\mathbb R)\cap {\mathcal L}^-_\alpha (\mathbb R) ) $$ and satisfies \eqref{eq1}
pointwise in  $\mathbb R^n \times \mathbb R$.

Our first result establishes the monotonicity of positive bounded solutions to the mixed local and nonlocal indefinite  elliptic equation \eqref{v1}.

\begin{theorem}\label{mthm12} Let $u(x )\in  C^2(\mathbb R^n)\cap {\mathcal L}_{2s}  $ is a positive bounded classical solution of equation \eqref{v1}.
Suppose that $u$ is uniformly continuous and  the following conditions hold:

\textnormal{(F1):}   $f$ is strictly increasing in $x_1$ and
\begin{equation}\label{p100}
\underset{x_1\rightarrow-\infty}{\lim \sup}f(x_1)|x_1|^{2s} \leq 0.
\end{equation}

\textnormal{(F2):}  $g$ is locally Lipschitz and nondecreasing in $(0,+\infty)$. Moreover, $ g'(u)$ is continuous near $u=0$,
 \begin{equation}\label{www}
 g(0)=0,~ g'(0)=0~\mbox{and}~ g(u)> 0 ~\mbox{ in}~(0,+\infty) .
\end{equation}
 Then   $u(x )$ is strictly monotone increasing in the $x_1$ direction.

\end{theorem}

\begin{remark}
 As for the nonlinearity $g$, one natural example is $g(t)=t^p$ with $1<p<\infty$. For the nonlinearlity $f$, a typical example is 
\begin{equation*}
f(x_1)=x^m_1 ,~m ~~\mbox{is an odd positive integer number}.
 \end{equation*}
 In \textnormal{(F1)}, $f(x_1)$ is not required to be sign-changing; for example, $f(x_1) = e^{x_1}$ is also admissible.
 Moreover,  there is no assumption that 
 $f(x_1)$  approaches  infinity as $x_1$ tends to infinity.
\end{remark}

The monotonicity result in Theorem \ref{mthm12} enables us to prove the  Liouville theorem  for mixed local and nonlocal equation \eqref{v1} with indefinite nonlinearities.

\begin{theorem}\label{thmf10}
Let the assumptions in Theorem \ref{mthm12} hold. In addition, assume that
\begin{equation}\label{e13}
 f(x_1) \rightarrow +\infty ~\mbox{as}~x_1 \rightarrow +\infty.
 \end{equation}
 Then  equation \eqref{v1} admits no positive bounded solution.
\end{theorem}

\begin{remark}
    It is also worth noting that the assumptions imposed in Theorems \ref{mthm12} and \ref{thmf10} are less restrictive than those in the fractional indefinite case (see \cite{CLZ, BPGQ2017}) and the classical indefinite case (see \cite{DLi}). A key feature of Theorems \ref{mthm12} and \ref{thmf10} is that the function $f(x_1)$  is not required to be nonpositive for $x_1<0$ or  positive for $x_1>0$ in Theorem  \ref{thmf10}. 
\end{remark}

We further consider the parabolic setting, where we establish monotonicity results for mixed local and nonlocal equations involving indefinite nonlinearities.

\begin{theorem}\label{mthm1} Let $0<s<1,~0<\alpha\leq 1,~u(x, t)\in ( C^2(\mathbb R^n)\cap \tilde { \mathcal L}_{2s})\times (C^1(\mathbb R)\cap {\mathcal L}^-_\alpha (\mathbb R)) $ be a positive bounded classical solution of equation \eqref{eq1}.
Suppose that  $u$ is uniformly continuous, $f$ satisfies condition $(F1)$ and $g$ satisfies condition $(F1)$,  
 then for each $t\in \mathbb R$,  $u(x, t)$ is strictly monotone increasing along $x_1$ direction.
\end{theorem}

The monotonicity result allows us to construct an explicit sub-solution, which in turn leads to the Liouville theorem for the mixed local and nonlocal indefinite parabolic equation \eqref{eq1}.

\begin{theorem}\label{mthm2}
Let the assumptions in Theorem \ref{mthm1} hold, suppose further that
$$\underset{x_1\rightarrow +\infty}{\lim}f(x_1)=+\infty.$$ Then equation \eqref{eq1} possesses no positive bounded classical solution.
\end{theorem}

\begin{remark}
    Theorems \ref{mthm1} and \ref{mthm2} apply to both the classical first-order  time derivative ($\alpha =1$) and   Marchaud fractional derivative ($0<\alpha<1$). When $0<\alpha<1$, equation \eqref{eq1} exhibits double non-locality due to the presence of both the temporal fractional derivative  $\partial_t^\alpha$ and  the spatial fractional derivative $(-\Delta)^s$.    This requires taking into account not only the double non-locality but also the local effects when analyzing equation \eqref{eq1}.
In addition, compared with the fractional indefinite parabolic equations involving the Marchaud fractional time derivative or the classical fractional indefinite parabolic equations (see \cite{CG, CWW}), there is no requirement that 
 $f(x_1)\rightarrow -\infty$  when $x_1 \to -\infty$. 
\end{remark}

To illustrate the main ideas employed in this paper for establishing monotonicity and the Liouville theorem, we first introduce some commonly used notation.

For any given $\lambda \in \mathbb{R},$ let
$$
T_\lambda=\{ x \in \mathbb{R}^n \mid x_1 =\lambda,\,\, \mbox{for some}\,\, \lambda \in \mathbb{R} \}
$$
denote the moving planes. Define
$$
\Sigma_\lambda:= \{ x \in \mathbb{R}^n \mid x_1 <\lambda \},
$$
as the region to the left of the plane, and
$$
x^\lambda=(2\lambda-x_1, x_2, ..., x_n)
$$
as the reflection of $x$ about the plane $T_\lambda.$

Let $u(x)$ or $u(x,t)$ denote  a positive solution of equation \eqref{v1} or \eqref{eq1}, respectively. In order to establish 
 the monotonicity of the solutions, we 
 compare the values of $u(x)$with $ u(x^\la)$ (or $u(x,t)$ with $u(x^\la,t)$), 
To facilitate the analysis, we focus on the difference
$$w_\la(x)=u(x^\la)-u(x)$$
and 
$$w_\la(x,t)=u(x^\la,t)-u(x,t).$$
To establish the monotonicity in the $x_1$ direction, it suffices to show that  $w_\lambda(x)$ or $w_\lambda(x, t)$ is nonnegative for all $x\in \Sigma_\lambda$  and 
$t \in \mathbb{R}$ and for any fixed $\lambda \in \mathbb{R}.$ 
 Essentially, our main task is to establish certain forms of maximum principles in unbounded domains. The primary challenge stems from the fact that, in the absence of decay conditions at infinity for $u(x)$ or $u(x, t)$, which cannot prevent the minimizing sequence of antisymmetric function $w_\la$ from leaking to infinity. As a consequence, the maxima or minima may fail to  be achieved. 
Clearly, due to the  presence of $f(x_1)$ and the mixed local and nonlocal operator, the coefficients of the transformed equation lack the required monotonicity. Therefore, the Kelvin transformation is not applicable.
For classical elliptic indefinite problem or fractional elliptic indefinite problem (\cite{CLZ, DLi}),  one can construct an auxiliary function 
$$ \bar{w}_\lambda(x)=\frac{w_\lambda(x)}{h(x)} $$
with 
\begin{equation*}
h(x)=|x-(\lambda+1)e_1|^\sigma >0,\,\,  \mbox{for some} \,\, \sigma>0.
\end{equation*}
With this choice, we have  $\bar w(x) \to 0$ as $|x| \to \infty$,  which provides the necessary decay to carry out the subsequent analysis.
However, in the case of the mixed local and nonlocal operator, applying it to  $\bar w (x)$
 does not yield suitable estimates due to the difference in order between  $(-\Delta)^s$ and $-\Delta$.

 To overcome this difficulty, we instead subtract a sequence of cutoff functions, thereby ensuring that the modified auxiliary functions attain their minima. By carefully estimating the singular integrals associated with the mixed local and nonlocal  operator acting on these auxiliary functions, we are able to derive a contradiction if 
$w_\lambda$ is negative somewhere in $\Sigma_\lambda$. 
 In the first step, we show that $w_\lambda(x) \geq 0$ when $\lambda$ is sufficiently negative.
To construct a sequence ${x^k}$ approaching the minima of $w_\lambda$, it is crucial to prove that the distance $r_k := \operatorname{dist}(x^k, T_\lambda)$ is bounded away from zero, due to the different rescaling behaviors of $(-\Delta)^s$ and $-\Delta$ within the ball $B_{r_k}(x^k)$.
In the second step, we continuously move the plane $T_\lambda$ to the right, ensuring that $w_\lambda(x) \geq 0$ for all $x \in \Sigma_\lambda$ until reaching the limiting position.
To this end, we construct a sequence ${x^k}$ approaching the minima of $w_{\lambda_k}$. Since $w_{\lambda_k}$ depends on $k$, it is necessary not only to show that $r_k$ is bounded away from zero, but also that it remains bounded above.
Due to the mismatch in the scaling properties of $(-\Delta)^s$ and $-\Delta$, some delicate estimates and new ideas are required to address this difficulty. To this end, we construct suitably modified auxiliary functions in two distinct regions—$B_{d_k^s}(x^k) \cap T_{\lambda_k}$ and $B_{d_k}(x^k)$—to derive estimates for $r_k$, see Step~2 in the proofs of Theorems~\ref{mthm12} and~\ref{mthm1} for details. The techniques developed in this paper can be readily adapted to investigate the qualitative properties of solutions to other mixed local and nonlocal elliptic or parabolic problems.



The remainder of the paper is organized as follows. In Section 2, we establish several key tools used throughout the paper, including maximum (and minimum) principles for mixed local and nonlocal operators in unbounded domains, as well as an integration by parts formula for the mixed operator. In Section 3, we apply the method of moving planes to prove the monotonicity of solutions and construct a subsolution to derive the nonexistence of positive solutions, thereby establishing Theorems \ref{mthm12} and \ref{thmf10}. Section 4 is devoted to proving the monotonicity of solutions in the $x_1$-direction and to establishing Liouville-type theorems for mixed local and nonlocal parabolic problems with indefinite nonlinearities, thereby proving Theorems \ref{mthm1} and \ref{mthm2}.

\section{Preliminaries }
In this section,  we establish  several fundamental tools used in proving the monotonicity and Liouville theorems for mixed local and nonlocal equations. These include the maximum principle for $u(x)$ in bounded domains (Lemma \ref{lemf4}), the minimum principle for antisymmetric functions in bounded domains (Lemma \ref{lem3}), and an integration by parts formula (Lemma \ref{mollification}).

\begin{lemma}\label{lemf4}(Maximum principle)  Let ~$\Omega$ be a bounded open domain.  Assume that $u(x)\in C^{2}(\Omega)\cap { \mathcal L}_{2s}$  is a lower semi-continuous function in $\bar{\Omega} $ and satisfies
\begin{equation}\label{ppp}\left\{
\begin{array}{ll}(-\Delta)^{s}u(x)-\lap u(x)\geq0, &\mbox~x\in \Omega,\\ u(x)\geq0, &\mbox~x\in {\mathbb R^{n}}  \setminus\Omega.
\end{array}
\right.\end{equation}

${\rm (i)}$ Then  \be \label{5201}
u(x)\geq0, ~x\in  {\mathbb R^{n}}.
\ee

${\rm (ii)}$
Under the conclusion \eqref{5201}, if $u  = 0$ at some point in
$\Omega$, then
$u(x) = 0 $ almost everywhere in $\mathbb R^n.$

${\rm (iii)}$ Furthermore, \eqref{5201} holds   for an unbounded domain $\Omega$ if we further assume that
$$\varliminf_{\left| x \right| \to \infty } {u }(x) \geq 0.$$
   \end{lemma}
\begin{remark}
Conclusions (i) and (ii) were previously established by Biagi et al. \cite{BDVV}, assuming that $u \in C(\mathbb{R}^n) \cap C^2(\Omega) \cap \mathcal{L}_{2s}$. For completeness and the reader’s convenience, we include a more detailed proof below.
\end{remark}

\begin{proof}[Proof of Lemma \ref{lemf4}]
(i) Suppose that \eqref{5201} does not hold. Then, by the lower semi-continuity of \( u \) on \( \bar{\Omega} \), there exists a point \( \bar{x} \in \bar{\Omega} \) such that
\[
u(\bar{x}) = \min_{\bar{\Omega}} u(x) < 0.
\]
Since \( u(x) \geq 0 \) for all \( x \in \mathbb{R}^n \setminus \Omega \), it follows that \( \bar{x} \in \Omega \); that is, \( \bar{x} \) is an interior minimum point of \( u \) in \( \Omega \). Hence, \( \Delta u(\bar{x}) \geq 0 \). A direct computation then yields
\[
\begin{aligned}
(-\Delta)^s u(\bar{x}) - \Delta u(\bar{x}) 
&\leq C_{n,s}~\text{P.V.} \int_{\mathbb{R}^n} \frac{u(\bar{x}) - u(y)}{|\bar{x} - y|^{n+2s}} \, dy \\
&\leq C_{n,s}~\text{P.V.} \int_{\mathbb{R}^n \setminus \Omega} \frac{u(\bar{x}) - u(y)}{|\bar{x} - y|^{n+2s}} \, dy \\
&< 0,
\end{aligned}
\]
which contradicts the assumed inequality \eqref{ppp}. Therefore, \eqref{5201} must hold.

(ii) Based on \eqref{5201}, suppose there exists a point \( \bar{x} \in \Omega \) such that
\[
u(\bar{x}) = \min_{\bar{\Omega}} u(x) = 0.
\]
Then,
\[
0 \leq (-\Delta)^s u(\bar{x}) - \Delta u(\bar{x}) \leq C_{n,s}~\text{P.V.} \int_{\mathbb{R}^n} \frac{-u(y)}{|\bar{x} - y|^{n+2s}} \, dy,
\]
and since \( u \geq 0 \) in \( \mathbb{R}^n \), we conclude that
\[
u(x) = 0 \quad \text{almost everywhere in } \mathbb{R}^n.
\]

(iii) If \( \Omega \) is unbounded, the condition \( \lim_{|x| \to \infty} u(x) \geq 0 \) guarantees that any negative minimum must be attained at some point. One can then repeat the argument in part (i) to derive a contradiction.

This completes the proof of Lemma \ref{lemf4}.
\end{proof}

\begin{lemma}\label{lem3}
(Minimum principle for antisymmetric functions in unbounded domains) Let $\Omega $ be an  unbounded open domain in $\Sigma_\la$ with $\la$ sufficiently negative.
Assume  that $ W_\la(x)=u(x)-u_\la(x) \in   C^2(\Omega)\cap \mathcal {L}_{2s}$ is bounded and  uniformly continuous. If
\begin{equation}\label{pp}
    \underset{x_1\rightarrow-\infty}{\lim \sup}f(x_1)|x_1|^{2s} \leq 0,
\end{equation} 
and
\begin{equation}\label{diffe}
\left\{\aligned
&(-\Delta)^sW_\lambda(x)-\lap W_\la(x) -C_0(x)f(x_1) W_\lambda(x) \leq 0,  &   x \in   \Omega,\\
&W_\lambda(x  )\leq 0,  &  x  \in   \Sigma_{\lambda}  \setminus \Omega,\\
\endaligned \right.
\end{equation}
where $C_0(x)$ is a  bounded and nonnegative function,
 then  we have
\begin{equation}\label{v11}
W_\la(x)\leq 0,~x\in \Omega.
\end{equation}
\end{lemma}
\begin{remark}
In the case  $C_0(x)=0$, the assumption that 
$\lambda$ is sufficiently negative and condition \eqref{pp} are not required, while the conclusion \eqref{v11} remains valid.
\end{remark}

\begin{proof}
[Proof of Lemma \ref{lem3}.]
Suppose \eqref{v11} does not hold, then
$$A:=\underset{\Omega}{\sup}~W_\la(x)>0.$$
Thus there exists a sequence $\{x^k\} \subset \Omega$ such that
\begin{equation}\label{mit}
W_\la(x^k)\rightarrow A,~\mbox{as}~k\rightarrow +\infty.
\end{equation}

Let $$ \zeta(x)=\left\{\begin{array}{lll} M e^{\frac{1}{|x|^2-1}},& |x|< 1,\\
 0,&   |x|\geq 1,
  \end{array}\right.$$
  where $M$ is a positive constant such that $\underset{\mathbb R^n}{\max}~\zeta(x)=1$. Obviously, $\zeta(x)$ is radially decreasing with respect to $|x|$, and its support is $B_1(0)$. We rescale $\zeta(x)$ by setting
  $$ \Psi_k(x)=\zeta(\frac{x-x^k}{r_k}),$$
  where $r_k=\frac{dist(x^k,T_\la)}{2}.$ 
  
  We claim that $r_k$ is bounded away from zero uniformly, that is, there exists a constant
  $\hat C$ such that
   \begin{equation}\label{88c1}
   r_k=\frac{dist(x^k,T_\la)}{2}\geq \hat C>0,~k\in \mathbb N.
   \end{equation}
   Otherwise, $\{x^k\}$ will  converge to a point on $T_\la$.  Then, by the uniform continuity of $W_\lambda$, we have
   \begin{equation*} 
    W_\la(x^k)\rightarrow 0, \,\,\mbox{as}\,\, k\to \infty.
   \end{equation*}
  This contradicts  assumption \eqref{mit}. Hence, \eqref{88c1} holds.

To proceed with the proof,  we chose a sequence $\varepsilon_k>0$ with $\varepsilon_k \to 0$  such that
  $$ W_\la(x^k)+\varepsilon_k\Psi_k(x^k)>A\geq W_\la(x)+\varepsilon_k\Psi_k(x),~\forall~x\in \Sigma_\la  \backslash B_{r_k}(x^k).$$
  Hence, there exists a point $ \bar x^k \in B_{r_k}(x^k) \subset \Omega$ satisfying
  $$ W_\la(\bar x^k)+\varepsilon_k\Psi_k(\bar x^k)=\underset{\Sigma_\la}{\max} (W_\la( x )+\varepsilon_k\Psi_k(  x ))>A.$$
On the other hand, a straightforward calculation shows that
 \begin{equation}\label{c1}\aligned
 &(-\lap)^s (W_\la(\bar x^k)+\varepsilon_k\Psi_k(\bar x^k) )-\lap (W_\la(\bar x^k)+\varepsilon_k\Psi_k(\bar x^k))\\
 \geq &C_{n,s}P.V.\int_{\mathbb R^n} \frac{W_\la(\bar x^k)+\varepsilon_k\Psi_k(\bar x^k)-W_\la(y)-\varepsilon_k\Psi_k(y)}{|\bar x^k -y|^{n+2s}} dy\\
 =&C_{n,s}P.V.\int_{\Sigma_\la} \frac{W_\la(\bar x^k)+\varepsilon_k\Psi_k(\bar x^k)-W_\la(y)-\varepsilon_k\Psi_k(y)}{|\bar x^k -y|^{n+2s}} dy\\
 &+C_{n,s} \int_{\Sigma_\la} \frac{W_\la(\bar x^k)+\varepsilon_k\Psi_k(\bar x^k)-W_\la(y^\la)-\varepsilon_k\Psi_k(y^\la)}{|\bar x^k -y^\la|^{n+2s}} dy\\
 \geq &C_{n,s} \int_{\Sigma_\la} \frac{2(W_\la(\bar x^k)+\varepsilon_k\Psi_k(\bar x^k)) -\varepsilon_k\Psi_k(y )}{|\bar x^k -y^\la|^{n+2s}} dy\\
 \geq &C_{n,s} C \int_{\Sigma_\la} \frac{2(W_\la(\bar x^k)+\varepsilon_k\Psi_k(\bar x^k)) -\varepsilon_k\Psi_k(y )}{| x^k -y^\la|^{n+2s}} dy\\
  \geq &C_{n,s}C\int_{\Sigma_\la}\frac{2A-\varepsilon_k}{| x^k -y|^{n+2s}} dy\\
 \geq &\frac{C_{n,s} C(2A-\varepsilon_k)}{r_k^{2s}}\int_{\mathbb R^n \backslash\Sigma_2} \frac{1}{|z|^{n+2s}} dz\\
 \geq & \tilde C \frac{2A-\varepsilon_k }{r_k^{2s}},
 \endaligned\end{equation}
 where  $z=\frac{x^k-y^\la}{r_k},$ and we have used the fact that 
 $$ |\bar x^k-y^\la|\leq |x^k-y^\la|+|\bar x^k-x^k|\leq \frac{3}{2}|x^k-y^\la|.$$

On the  other hand, recalling that $\underset{\Omega}{\sup}~W_\la(x)=\underset{\Sigma_\lambda}{\sup}~W_\la(x)=A>0$
 and $r_k\geq \hat C>0$, 
 we can deduce from \eqref{diffe} that 
   \begin{equation}\label{c3}\aligned
 &(-\lap)^s (W_\la(\bar x^k)+\varepsilon_k\Psi_k(\bar x^k) )-\lap (W_\la(\bar x^k)+\varepsilon_k\Psi_k(\bar x^k))\\
 \leq & C_0(x)f(\bar x^k_1)W_\la(\bar x^k)+ \varepsilon_k(-\lap)^s \Psi_k(\bar x^k)-\varepsilon_k \lap \Psi_k(\bar x^k)\\
 \leq &C_0(x)f(\bar x^k_1) W_\la(\bar x^k)+\frac{ c_1\varepsilon_k}{r_k^{2s}}+\frac{ c_2\varepsilon_k}{r_k^{2}}\\
 \leq & C_0(x)f(\bar x^k_1) A+  \frac{ c_3\varepsilon_k}{r_k^{2s}}.
 \endaligned\end{equation}
 Combining  \eqref{c1} with \eqref{c3}, we obtain
 $$\tilde C  \frac{2A-\varepsilon_k}{r_k^{2s}}\leq C_0(x)f(\bar x^k_1) A  +\frac{c_3\varepsilon_k}{r_k^{2s}}  ,$$
 therefore,
 $$ 0< 2A \tilde C    \leq C_0(x)f(\bar x^k_1) A r_k^{2s}+ \tilde C \varepsilon_k +c_3\varepsilon_k \leq C_0(x)A f(\bar x^k_1) |\bar x^k_1|^{2s}+ \tilde C \varepsilon_k +c_3\varepsilon_k.$$
Since $\lambda$ is sufficiently negative, 
letting $\varepsilon_k\rightarrow 0$ when $k\rightarrow \infty$ and using \eqref{pp},   we arrive at a contradiction. Therefore, \eqref{v11} is established and the proof of Lemma \ref{lem3} is complete. 
\end{proof}

\begin{lemma}\label{mollification}
Assume that $\varphi$ is the first eigenfunction associated with $(-\Delta)^s-\lap$ in $B_1(0)$ and satisfies
\begin{eqnarray*}
\left\{\begin{array}{ll}
 ( - \Delta )^s \varphi (x)-\lap \varphi (x) =\lambda_1  \varphi (x),& x \in B_1(0),\\
\varphi (x)=0,& x \in B^c_1(0).
\end{array} \right.
\end{eqnarray*}
Let $\rho  \in C^\infty_0(B_1(0))$ be a standard mollifier such that  $\int_{B_1(0)} \rho (x)dx=1$.  Then, the  following identity holds
\begin{eqnarray}\label{eq2-mthm}
\int_{\mathbb{R}^n} ((-\Delta)^s -\lap) \varphi(z) \rho  (x-z)dz =\int_{\mathbb{R}^n} \varphi(z)  ((-\Delta)^s -\lap) \rho  (x-z)dz,
\end{eqnarray}
and
\begin{eqnarray}\label{inter}
(( - \Delta ) ^s-\lap) \varphi_1 (x) \leq \lambda_1  \varphi_1 (x),\,\, x \in \mathbb R^n,
\end{eqnarray}
where  $\varphi_1 (x)=(\rho \ast \varphi)(x)= (\varphi \ast \rho ) (x)$  denotes the mollification of $\varphi$.
\end{lemma}

\begin{proof}
 We first estimate  $((-\Delta)^s-\lap)\varphi(x)$ in $ B_1^c(0),$ that is,  we  show that
\begin{eqnarray}\label{eq2-16}
\left| ((-\Delta)^s -\lap)\varphi(x) \right| \leq C\frac{1}{dist^s(x, \partial B)},\,\, x\in B_1^c(0).
\end{eqnarray}
According to \cite[Theorem B.1]{SVWZ},  the eigenfunction satisfies $\varphi(x)\in C^{2,\kappa}_{loc}(B_1(0))\cap C^{1,\kappa}(\overline {B_1(0)})$ for any  $\kappa\in(0,1)$.
Now for $x \in \mathbb{R}^n \backslash \overline {B_1(0)}, $ since $\varphi(x)=0$, we have 
\begin{eqnarray}\label{eq2-14}
\left |( (-\Delta)^s-\lap) \varphi(x)\right|
&=& C_{n, s}\left | PV \int_{\mathbb{R}^n}\frac{0-\varphi(y)}{|x-y|^{n+2s}}dy\right|\nonumber\\
&=&C_{n, s}\left | \int_{B_1(0)}\frac{\varphi(y)}{|x-y|^{n+2s}}dy\right|\nonumber\\ 
&\leq &C \int_{B_1(0)}\frac{(dist(y, \partial B))^{s}}{|x-y|^{n+2s}}dy\nonumber\\
&\leq & C \int_{B_1(0)}\frac{1}{|x-y|^{n+s}}dy.
\end{eqnarray}
Next, we estimate the last integral in \eqref{eq2-14}.
Denote $x=(x_1, 0'), \,\, x_1<0,$ and 
$$
D=\{y \mid 0 < y_1< 2,\,\, |y'|<1 \}.
$$
We then obtain 
\begin{eqnarray}\label{eq2-15}
\int_{B_1(0)}\frac{1}{|x-y|^{n+s}}dy && =\int_{B_1(1,0') }\frac{1}{|x-(y_1-1,y')|^{n+s}} dy \nonumber\\
&\leq& \int_{D} \frac{1}{|x-(y_1-1,y')|^{n+s}} dy\nonumber\\
&=& C \int_0^2 \int_0^1 \frac{r^{n-2}}{(r^2+(y_1-x_1-1)^2)^{\frac{n+2s-1-\kappa}{2}}}drdy_1\nonumber\\
&=& C \int_0^2 \frac{1}{|y_1-x_1-1|^{s+1}}dy_1\nonumber\\
& \leq  & C \frac{1}{|x_1|^{s}}.
\end{eqnarray}
Combining \eqref{eq2-14} with \eqref{eq2-15}, we arrive at \eqref{eq2-16}.

Secondly, we prove \eqref{eq2-mthm}. For simplicity, let
$v(z)= \rho (x-z).$ Therefore, we need to show that
\begin{eqnarray*}
\int_{\mathbb{R}^n} ((-\Delta)^s -\lap)\varphi(x) v(x)dx =\int_{\mathbb{R}^n} \varphi(x)  ((-\Delta)^s-\lap) v(x)dx.
\end{eqnarray*}
By the definition of the fractional Laplacian, the above equality is equivalent to
\begin{eqnarray}\label{eq2-17}\aligned
&\int_{\mathbb{R}^n}  v(x)\big[\lim_{\epsilon \to 0}\int_{|y-x|\geq \epsilon} \frac{\varphi(x)-\varphi(y)}{|x-y|^{n+2s}}dy-\lap \varphi(x)\big] dydx \\
=&\int_{\mathbb{R}^n} \varphi(x) \big[ \lim_{\epsilon \to 0}\int_{|y-x|\geq \epsilon}\frac{v(x)-v(y)}{|x-y|^{n+2s}}dy-\lap v(x)\big] dydx.
\endaligned\end{eqnarray}
Hence, it suffices to prove \eqref{eq2-17} in this step.

We analyze the two integrals on both sides of
  of \eqref{eq2-17} respectively. We prove that the limit as $\epsilon \rightarrow 0$ can be interchanged with the integral over $y$.

We first consider the integral on the left-hand side of \eqref{eq2-17} in $B^c_1(0)$ and $B_1(0)$ respectively. If $x \in \mathbb{R}^n \backslash \overline{B_1(0)},$ by \eqref{eq2-16}, we derive
$$
\left| \int_{|y-x|\geq \epsilon} \frac{\varphi(x)-\varphi(y)}{|x-y|^{n+2s}}dy-\lap \varphi(x) \right|=\left| \int_{|y-x|\geq \epsilon} \frac{\varphi(x)-\varphi(y)}{|x-y|^{n+2s}}dy  \right| \leq \frac{C}{dist^s(x, \partial B_1(0))}.
$$
By Lebesgue's dominated convergence theorem, we have
\begin{eqnarray}\label{eq2-19}\aligned
&\int_{B_1^c(0)} v(x) \big[\lim_{\epsilon \to 0}\int_{|y-x|\geq \epsilon} \frac{\varphi(x)-\varphi(y)}{|x-y|^{n+2s}}dy-\lap \varphi (x) \big]dx\\
=&\lim_{\epsilon \to 0} \int_{B_1^c(0)}v(x) \int_{|y-x|\geq \epsilon} \frac{\varphi(x)-\varphi(y)}{|x-y|^{n+2s}}dy dx-\int_{B_1^c(0)} v(x)\lap \varphi(x)dx.
\endaligned\end{eqnarray}

  If $x \in B_{1-\delta}(0),$ for any fixed $0<\delta <1$, by $\varphi(x)\in C^{2,\kappa}_{loc}(B_1(0))\cap C^{1,\kappa}(\overline{B_1(0)})$, we obtain
$$
\left| \int_{|y-x|\geq \epsilon} \frac{\varphi(x)-\varphi(y)}{|x-y|^{n+2s}}dy-\lap \varphi(x) \right| \leq C \|\varphi\|_{C^{1,1}(B_{1-\delta}(0))}.
$$
Similarly, by Lebesgue's dominated convergence theorem, one has
\begin{eqnarray*}\aligned
&\int_{B_{1-\delta}(0)} v(x)\big[\lim_{\epsilon \to 0}\int_{|y-x|\geq \epsilon} \frac{\varphi(x)-\varphi(y)}{|x-y|^{n+2s}}dy -\lap \varphi (x)\big]dx\\
=&\lim_{\epsilon \to 0} \int_{B_{1-\delta}(0)}v(x) \int_{|y-x|\geq \epsilon} \frac{\varphi(x)-\varphi(y)}{|x-y|^{n+2s}}dy dx-\int_{B_{1-\delta}(0)}v(x) \lap \varphi(x)dx.
\endaligned\end{eqnarray*}
Let $\delta \to 0,$ we obtain
\begin{eqnarray}\label{eq2-20}\aligned
&\int_{B_1(0)} v(x)\big [\lim_{\epsilon \to 0}\int_{|y-x|\geq \epsilon} \frac{\varphi(x)-\varphi(y)}{|x-y|^{n+2s}}dy-\lap \varphi(x) \big]dx\\
=&\lim_{\epsilon \to 0} \int_{B_1(0)}v(x) \int_{|y-x|\geq \epsilon} \frac{\varphi(x)-\varphi(y)}{|x-y|^{n+2s}}dy dx-\int_{B_1(0)} v(x)\lap \varphi(x) dx.
\endaligned\end{eqnarray}

Combining \eqref{eq2-19} with \eqref{eq2-20}, we obtain that the limit as $\epsilon\rightarrow 0$ can be interchanged with the integral over $x$ on the left-hand side of \eqref{eq2-17}, more precisely,
\begin{eqnarray}\label{eq2-21}\aligned
&\int_{\mathbb{R}^n}v(x) \big[\lim_{\epsilon \to 0}\int_{|y-x|\geq \epsilon} \frac{\varphi(x)-\varphi(y)}{|x-y|^{n+2s}}dy-\lap \varphi(x)\big] dx\\
=&\lim_{\epsilon \to 0} \int_{\mathbb{R}^n}v(x) \int_{|y-x|\geq \epsilon} \frac{\varphi(x)-\varphi(y)}{|x-y|^{n+2s}}dydx- \int_{\mathbb{R}^n}v(x)\lap \varphi(x)dx.
\endaligned\end{eqnarray}

Then  we consider the integral on the right-hand side of \eqref{eq2-17}.  By Lebesgue's dominated convergence theorem, it is easy to see that the limit as $\epsilon\rightarrow 0$ and the integral over $x$ on the right-hand side of \eqref{eq2-17} can be interchanged, $i.e.$,
\begin{eqnarray}\label{eq2-18}\aligned
&\int_{\mathbb{R}^n} \varphi(x) \big[ \lim_{\epsilon \to 0}\int_{|y-x|\geq \epsilon}\frac{v(x)-v(y)}{|x-y|^{n+2s}}dy-\lap v(x)\big]dx\\
=&\lim_{\epsilon \to 0} \int_{\mathbb{R}^n} \varphi(x)  \int_{|y-x|\geq \epsilon}\frac{v(x)-v(y)}{|x-y|^{n+2s}}dydx-\int_{\mathbb R^n}\varphi(x)\lap v(x)dx.
\endaligned\end{eqnarray}

To conclude the proof of \eqref{eq2-17},  by $ \int_{\mathbb R^n}v(x)\lap \varphi(x)dx=\int_{\mathbb R^n}\varphi(x)\lap v(x)dx$, \eqref{eq2-21} and \eqref{eq2-18}, it remains to show that
$$\int_{\mathbb{R}^n} v(x)  \int_{|y-x|\geq \epsilon}\frac{\varphi(x)-\varphi(y)}{|x-y|^{n+2s}}dydx- \int_{\mathbb{R}^n} \varphi(x)  \int_{|y-x|\geq \epsilon}\frac{v(x)-v(y)}{|x-y|^{n+2s}}dydx=0.$$
This follows  from Fubini's Theorem that
\begin{eqnarray*}
 &&\int_{\mathbb{R}^n} v(x)  \int_{|y-x|\geq \epsilon}\frac{\varphi(x)-\varphi(y)}{|x-y|^{n+2s}}dydx- \int_{\mathbb{R}^n} \varphi(x)  \int_{|y-x|\geq \epsilon}\frac{v(x)-v(y)}{|x-y|^{n+2s}}dydx\\
 &=&\int_{\mathbb{R}^n}\int_ {|y-x|\geq\epsilon}\frac{\varphi(x)v(x)-\varphi(y)v(x)}{\mid x-y\mid^{n+2s}}dydx-\int_{\mathbb{R}^n}\int_ {|y-x|\geq\epsilon}\frac{\varphi(x)v(x)-\varphi(x)v(y)}{\mid x-y\mid^{n+2s}}dydx\\
 &=&\int_{\mathbb{R}^n}\int_ {|y-x|\geq\epsilon}\frac{-\varphi(y)v(x)}{\mid x-y\mid^{n+2s}}dydx-\int_{\mathbb{R}^n} \int_ {|y-x|\geq\epsilon}\frac{-\varphi(x)v(y)}{\mid x-y\mid^{n+2s}}dydx\\
 &=&-\int_{\mathbb{R}^n}\int_ {|y-x|\geq\epsilon}\frac{\varphi(y)v(x)}{\mid x-y\mid^{n+2s}}dydx+\int_{\mathbb{R}^n} \int_ {|y-x|\geq\epsilon}\frac{\varphi(y)v(x)}{\mid x-y\mid^{n+2s}}dydx\\
 &=&0.
 \end{eqnarray*}
Hence, \eqref{eq2-17} holds and \eqref{eq2-mthm} is verified.

Finally, we prove \eqref{inter}.

By the definitions of the fractional Laplacian and Laplacian, and the mollifier, we obtain
\begin{eqnarray*}
&&(( - \Delta ) ^s -\lap )\varphi_1 (x)\\
&=&C_{n,s} PV \int_{\mathbb{R}^n} \frac{\int_{\mathbb{R}^n}\rho (x-z)\varphi(z)dz-\int_{\mathbb{R}^n}\rho (y-z)\varphi(z)dz}{|x-y|^{n+2s}}  dy-\lap \varphi_1(x)\\
&=& \int_{\mathbb{R}^n} \varphi(z) \big[(-\Delta)_x^s \rho (x-z)-\lap_x \rho(x-z)\big]dz\\
&=& \int_{\mathbb{R}^n} \varphi(z)\big[ (-\Delta)_z^s \rho (x-z)-\lap_z(x-z)\big]dz\\
&=& \int_{B_1(0)}  ((-\Delta)_z^s-\lap_z) \varphi(z) \rho (x-z)dz+\int_{\mathbb{R}^n \backslash B_1(0)}  ((-\Delta)_z^s-\lap_z) \varphi(z) \rho (x-z)dz\\
&\leq & \int_{B_1(0)}  ((-\Delta)_z^s-\lap_z) \varphi(z) \rho (x-z)dz\\
&= &\int_{B_1(0)}  \lambda_1 \varphi(z) \rho (x-z)dz\\
&=&\lambda_1 \varphi_1 (x).
\end{eqnarray*}
Therefore, we obtain \eqref{inter}, and  complete the proof of Lemma \ref{mollification}.
\end{proof}

\begin{remark}\label{mrem-inter}
It follows from the proof of the above lemma that if $u \in C^{2} (\mathbb{R}^n)\cap   {\mathcal L}_{2s } $  and $v\in C_0^\infty(\mathbb{R}^n),$ the following identity holds
$$
\int_{\mathbb{R}^n}((-\lap)^s-\lap) u(x)\, v(x)dx=\int_{\mathbb{R}^n} u(x)\, ((-\lap)^s -\lap)v(x)dx.
$$
\end{remark}

\section{Liouville theorem for  mixed local and nonlocal indefinite elliptic equations}

In this section, we aim to establish the monotonicity of positive solutions to the mixed local and nonlocal elliptic equation \eqref{v1} with indefinite nonlinearities. We then proceed to prove the Liouville theorem for positive solutions of \eqref{v1}, namely Theorems \ref{mthm12} and \ref{thmf10}. Before doing so, we first derive a key differential inequality.

\begin{lemma}\label{lem4}
Suppose that $w_\lambda < 0$ somewhere in $\Sigma_\lambda$, for some $\lambda \in \mathbb{R}$. Define
\[
v_\lambda(x) := w_\lambda(x)\, \chi_{E_\lambda}(x),
\]
where
\[
\chi_{E_\lambda}(x) :=
\begin{cases}
1, & \text{if } x \in E_\lambda, \\
0, & \text{if } x \notin E_\lambda,
\end{cases}
\quad \text{with} \quad E_\lambda := \{x \in \Sigma_\lambda \mid w_\lambda(x) < 0\}.
\]
Then, the following inequality holds:
\begin{equation}\label{1c00}
(-\Delta)^s v_\lambda(x) - \Delta v_\lambda(x) - \left[(-\Delta)^s w_\lambda(x) - \Delta w_\lambda(x)\right] \geq 0, \quad \text{for } x \in E_\lambda.
\end{equation}
\end{lemma}

\begin{proof}
Let $$V(x) = v_\lambda(x) - w_\lambda(x).$$ Clearly, \eqref{1c00} is equivalent to
\[
(-\Delta)^s V(x) - \Delta V(x) \geq 0 \quad \text{in } E_\lambda.
\]
Denote by $\Gamma_\lambda$ the reflection of the set $E_\lambda$ with respect to the hyperplane $T_\lambda$. Since $V \equiv 0$ in $E_\lambda$ and $V \geq 0$ in $\Sigma_\lambda \setminus E_\lambda$, we deduce for any $x \in E_\lambda$ that
\[
-\Delta V(x) = 0,
\]
and
\[
\begin{aligned}
(-\Delta)^s V(x) &= \int_{\mathbb{R}^n} \frac{-V(y)}{|x - y|^{n + 2s}}\,dy \\
&= \left( \int_{\Sigma_\lambda \setminus E_\lambda} + \int_{\Gamma_\lambda} + \int_{\Sigma_\lambda^c \setminus \Gamma_\lambda} \right) \frac{-V(y)}{|x - y|^{n + 2s}}\,dy \\
&= \int_{\Sigma_\lambda \setminus E_\lambda} \frac{-V(y)}{|x - y|^{n + 2s}}\,dy + \int_{E_\lambda} \frac{-V(y^\lambda)}{|x - y^\lambda|^{n + 2s}}\,dy \\
&\quad + \int_{\Sigma_\lambda^c \setminus E_\lambda} \frac{-V(y^\lambda)}{|x - y^\lambda|^{n + 2s}}\,dy \\
&= \int_{\Sigma_\lambda \setminus E_\lambda} (-V(y)) \left[ \frac{1}{|x - y|^{n + 2s}} - \frac{1}{|x - y^\lambda|^{n + 2s}} \right] dy \\
&\geq 0,
\end{aligned}
\]
where in the last inequality we have used the fact that $|x - y| < |x - y^\lambda|$ for all $x \in E_\lambda$ and $y \in \Sigma_\lambda$.

Therefore, inequality \eqref{1c00} holds, and the proof of Lemma~\ref{lem4} is complete.
\end{proof}

\begin{proof}
[Proof of Theorem \ref{mthm12}.] 
Let $u(x)$ be a positive bounded classical solution of equation \eqref{v1}, and define
\[
u_\lambda(x) = u(x^\lambda), \quad w_\lambda(x) = u_\lambda(x) - u(x), \quad x \in \mathbb{R}^n.
\]
We aim to prove that $u$ is strictly increasing in the $x_1$-direction, that is,
\[
w_\lambda(x) > 0, \quad \forall\, x \in \Sigma_\lambda,\ \text{for all } \lambda \in \mathbb{R}.
\]
The proof is divided into the following three steps.

\textit{Step 1}. We show that for $\lambda$ sufficiently negative,
\begin{equation}\label{eq2w}
w_\lambda(x) \geq 0, \quad \forall\, x \in \Sigma_\lambda.
\end{equation}
From assumptions \textnormal{(F1)} and \textnormal{(F2)}, and using equation \eqref{v1}, we have
\begin{equation}\label{eq2w1}
\begin{aligned}
(-\Delta)^s w_\lambda(x) - \Delta w_\lambda(x) &= f(x_1^\lambda)g(u_\lambda(x)) - f(x_1)g(u(x)) \\
&= \big(f(x_1^\lambda) - f(x_1)\big)g(u_\lambda(x)) + f(x_1)\big(g(u_\lambda(x)) - g(u(x))\big) \\
&\geq f(x_1)M_\lambda(x)w_\lambda(x),
\end{aligned}
\end{equation}
where 
\[
M_\lambda(x) = \frac{g(u_\lambda(x)) - g(u(x))}{u_\lambda(x) - u(x)}.
\]
Since $u$ is bounded and $g(\cdot)$ is locally Lipschitz continuous and nondecreasing, it follows that
\begin{equation}\label{eq2w3}
M_\lambda(x) \text{ is bounded and nonnegative.}
\end{equation}

We proceed by contradiction. Suppose that there exists $\lambda \leq 0$ such that $w_\lambda(x) < 0$ somewhere in $\Sigma_\lambda$. Then we can define the nonempty open set
\[
E_\lambda = \{x \in \Sigma_\lambda \mid w_\lambda(x) < 0\},
\]
and introduce the auxiliary function
$$v_\la(x)=w_\la(x) \chi_{E_\la(x)}\leq 0,$$
where $$\chi_{E_\la}(x) =\left\{\begin{array}{ll}
1,& \quad  x \in  E_\la,\\
 0,&\quad  x \not \in E_\la  .
\end{array}\right.$$

Obviously, for $x \in \Sigma_\lambda$ we have $x_1^\lambda > x_1$. Using the monotonicity of both $f$ and $g$, as well as the nonnegativity of $g$, it follows that for $x \in E_\lambda$, similarly to \eqref{eq2w1},  
\begin{equation}\label{eq:o110}
(-\Delta)^s w_\lambda(x) - \Delta w_\lambda(x) 
\geq  f(x_1) M_\lambda(x) w_\lambda(x):= C_0(x) f(x_1) w_\lambda(x),
\end{equation}
where $C_0(x)$ is bounded and nonnegative by \eqref{eq2w3}.

Employing Lemma \ref{lem4} and \eqref{eq:o110}, we obtain
\begin{equation*}
(-\Delta)^s v_\lambda(x) - \Delta v_\lambda(x) \geq (-\Delta)^s w_\lambda(x) - \Delta w_\lambda(x) \geq C_0(x) f(x_1) w_\lambda(x), \quad x \in E_\lambda.
\end{equation*}

Since $v_\lambda = 0$ for $x \in \mathbb{R}^n \setminus E_\lambda$, applying the maximum principle for the mixed local and nonlocal operator in unbounded domains (Lemma \ref{lem3}) yields $w_\lambda \geq 0$ in $E_\lambda$. This contradicts the assumption that $w_\lambda < 0$ somewhere in $\Sigma_\lambda$, verifying that $E_\lambda$ must be empty. Hence, $w_\lambda \geq 0$ in $\Sigma_\lambda$ for sufficiently negative $\lambda$, that is, \eqref{eq2w} holds.

\textit{Step 2}.  The above inequality provides a starting point
 for applying the method of moving planes.  Now we move
plane $T_\la$ towards the right as long as the inequality  continue to hold. Our goal is to show that the plane $T_\la$
can be moved all the way to infinity. More precisely, let
$$\la_0=\sup\{\la\in \mathbb R\mid w_\mu\geq 0~\mbox{in}~\Sigma_\mu~\mbox{for every}~\mu<\la\},$$
and we will prove that 
$$
\la_0=+\infty.$$

Assume, for contradiction, that $\lambda_0 < +\infty$. Then, by the definition of $\lambda_0$, there exists a sequence $\{\lambda_k\} \searrow \lambda_0$ and a corresponding positive sequence $\{m_k\}$ such that
\begin{equation}\label{e12}
  \inf_{\Sigma_{\lambda_k}} w_{\lambda_k}(x) := -m_k < 0.
\end{equation}
We first claim that
\begin{equation}\label{ee13}
 m_k \rightarrow 0 \quad \text{as } k \rightarrow +\infty.
\end{equation}
Suppose, on the contrary, that there exists a subsequence (still denoted by $\{m_k\}$) such that $m_k > \bar C$ for some positive constant $\bar C$. Then there exists a sequence $\{y^k\} \subset \Sigma_{\lambda_k}$ such that
\begin{equation}\label{e31}
w_{\lambda_k}(y^k) \leq -\bar C < 0.
\end{equation}

\textit{Case (i).} If $y^k \in \Sigma_{\lambda_k} \setminus \Sigma_{\lambda_0}$, then since $\lambda_k \rightarrow \lambda_0$ as $k \rightarrow \infty$, we have
\[
|y^k - (y^k)^{\lambda_k}| = 2|\lambda_k - y^k_1| \rightarrow 0 \quad \text{as } k \rightarrow \infty.
\]
By the uniform continuity of $u$, it follows that
\[
w_{\lambda_k}(y^k) = u((y^k)^{\lambda_k}) - u(y^k) \rightarrow 0 \quad \text{as } k \rightarrow \infty.
\]

\textit{Case (ii).} If $y^k \in \Sigma_{\lambda_0}$, then using the uniform continuity of $u$, the fact that $\lambda_k \rightarrow \lambda_0$, and the definition of $\lambda_0$, we obtain
\[
\begin{aligned}
w_{\lambda_k}(y^k) &= u((y^k)^{\lambda_k}) - u((y^k)^{\lambda_0}) + w_{\lambda_0}(y^k) \\
&\geq u((y^k)^{\lambda_k}) - u((y^k)^{\lambda_0}) \rightarrow 0 \quad \text{as } k \rightarrow \infty.
\end{aligned}
\]
In both cases, we arrive at a contradiction with \eqref{e31}. Therefore, the claim \eqref{ee13} must hold.
\vspace{0.4cm}

Secondly, 
from \eqref{ee13}, there exists a sequence $\{x^k\}\subset \Sigma_{\la_k}$ such that
\begin{equation}\label{112z}
 w_{\lambda_k}(x^k)=-m_k+m_k^2<0.
\end{equation}
we will show that 
\begin{equation}\label{22e26}
d_k=\frac{1}{2}dist\{x^k,T_{\la_k}\}\geq \bar C>0 .
\end{equation}
Without loss of generality, we assume $d_k<1$, and hence $d_k<d_k^s$.

To this end, we   we construct a perturbation of  $w_{\la_k}$ near $x^k$ as follows
\begin{equation}\label{2z}
V_k(x)=w_{\la_k}(x)-m_k^2\varsigma_k(x),~x\in \mathbb R^n,
\end{equation}
where $$\varsigma_k(x)=\varsigma\left(\frac{x-x^k}{d^s_k}\right),$$ and $\varsigma\in C^{\infty}_0(\mathbb R^n)$ is a cut-off smooth function satisfying
 $$\left\{\begin{array}{ll}
0\leq \varsigma \leq 1,& \mbox{in}~\mathbb R^n,\\
 \varsigma=1,&  \mbox{in}~B_{\frac{1}{2}}(0),\\
 \varsigma=0, & \mbox{in}~\mathbb R^n \backslash B_1(0) .
\end{array}
\right.$$

Combining \eqref{e12}, \eqref{112z}, and \eqref{2z}, we obtain
\begin{equation*}
\begin{cases}
V_{ k}(x^k)= w_{\la_k} (x^k) -m_k^2\varsigma_k(x^k)=-m_k, &\\
V_{ k}(x )=w_{\la_k} (x ) -m_k^2\varsigma_k(x )\geq  -m_k, &\mbox{in }~ \mathbb R^n \setminus B_{d_k^s}(x^k),\\
 V_{ k}(x )=  -m_k^2\varsigma_k(x )\geq  -m_k^2, &\mbox{on}~ T_{\la_k}.
\end{cases}
\end{equation*}
This implies that each $V_k$ must attain its minimum, which is no greater than $-m_k$, at some point $\bar x^k \in \overline{B_{d^s_k}(x^k) \cap \Sigma_{\la_k}} \subset \Sigma_{\la_k}$ (see Figure 1), that is to say,
\begin{eqnarray}\label{z1}
\exists~ \{ \bar x^k \} \subset \overline{B_{d^s_k}(x^k)\cap \Sigma_{\la_k}} ~s.t. ~-m_k-m^2_k\leq V_k(\bar x^k)=\underset{\Sigma_{\la_k}  }{\inf}~ V_k(x )\leq -m_k,
\end{eqnarray}
 \begin{center}
\begin{tikzpicture}[scale=1.5][node distance = 5mm]
\draw[thin][blue] (1.76,0.75 ) circle (1);
\path (1.76,0.75) [very thick,fill=blue]  circle(1pt) node at (1.7,0.9) [ font=\fontsize{10}{10}\selectfont][blue] {{$ x^k  $}};
\path (1.26,0.25) [very thick,fill=red]  circle(1pt) node at (1.2,0.35) [ font=\fontsize{10}{10}\selectfont][red] {{$ \bar x^k  $}};
\path (1.2,2.1) node at (1.76,2 ) [ font=\fontsize{10}{10}\selectfont][blue] {$  B_{d_k^s}(x^k )$};
\draw [->,  semithick] (-0.7 ,-0.65) -- (3.3,-0.65) node[right] [ font=\fontsize{10}{10}\selectfont] {  {$x_1$}};
\draw [semithick] (2.65,-0.9) -- (2.65,2.3 ) node[above] [ font=\fontsize{8}{8}\selectfont] {{$ T_{\lambda_k}$}};
\path (0.46 , 0.3) node at (0.47, 1.8) [ font=\fontsize{10}{10}\selectfont] {$ \Sigma_{\lambda_k}$} ;
    \begin{scope}
     \clip (1.76,0.75 ) circle (1);
     \foreach \y in {-0.25,-0.15,...,1.75}
        \draw [dashed][gray][very thin](0.76,\y) -- (2.65,\y);
    \end{scope}
\node [below=1mm, align=flush center,text width=88mm] at (1.2,-1.2)
        { $Figure ~1$.  The domain $ B_{d_k^s}(x^k)\cap \Sigma_{\la_k}$  
         }  ;
\end{tikzpicture}
\end{center}
 which indicates that 
\begin{eqnarray}\label{ep1}
- m_k\leq  {w}_{\lambda_k}(\bar x^k )\leq -m_k+m^2_k<0.
\end{eqnarray}

From \eqref{z1} and the definitions of the operators $(-\Delta)^s$ and $-\Delta$, together with the anti-symmetry of $w_{\lambda_k}$ in $ \Sigma_{\lambda_k}$, we deduce
\begin{eqnarray}\label{e25}
-\lap V_k(\bar x^k ) \leq 0,
\end{eqnarray}
and
\begin{equation}\label{e26} \aligned
(-\lap)^s V_k(\bar x^k )=&C_{n,s}P.V.\int_{\mathbb R^n} \frac{V_k(\bar x^k )
-V_k(y)}{|\bar x^k-y|^{n+2s}}dy\\
=&C_{n,s}P.V.\int_{\Sigma_{\la_k}} \frac{V_k(\bar x^k )
-V_k(y )}{|\bar x^k-y|^{n+2s}}dy
 +C_{n,s}\int_{\Sigma_{\la_k}} \frac{V_k(\bar x^k )
-V_k(y^{\la_k} )}{|\bar x^k-y^{\la_k}|^{n+2s}}dy\\
\leq& C_{n,s}\int_{\Sigma_{\la_k}} \frac{2V_k(\bar x^k )-V_k(y )-V_k(y^{\la_k} )}{|\bar x^k-y^{\la_k}|^{n+2s}}dy\\
=& C_{n,s}\int_{\Sigma_{\la_k}} \frac{2V_k(\bar x^k )+m_k^2 \varsigma_k(y)-m_k^2  \varsigma_k(y^{\lambda_k})}{|\bar x^k-y^{\la_k}|^{n+2s}}dy\\
\leq & 2C_{n,s}(V_k(\bar x^k )+m^2_k)\int_{\Sigma_{\la_k}}\frac{1}{|\bar x^k-y^{\la_k}|^{n+2s}}dy\\
\leq & 2 C_{n,s}(V_k(\bar x^k )+m^2_k)\int_{\Sigma_{\la_k}}\frac{1}{|  x^k-y^{\la_k}|^{n+2s}}dy\\
\leq &\frac{c_1(-m_k+m^2_k)}{d_k^{2s}},
 \endaligned\end{equation}
 where   the second inequality from the bottom follows from the estimate
 $$ |\bar x^k-y^{\la_k}|\leq |\bar x^k-x^k|+|x^k-y^{\la_k}|\leq \frac{3}{2}|x^k-y^{\la_k}|.$$
 Combining \eqref{eq2w1}, \eqref{2z},  \eqref{e25} with \eqref{e26},  we obtain
\begin{equation}\label{finalde}
\begin{aligned}
 \frac{c_1(-m_k+m^2_k)}{d_k^{2s}}&\geq (-\lap)^s V_k(\bar x^k )-\lap V_k(\bar x^k )\\
 &=(-\lap)^s w_{\la_k}(\bar x^k )-\lap w_{\la_k}(\bar x^k )-m_k^2(-\lap)^s\varsigma_k(\bar x^k)+m_k^2 \lap \varsigma_k(\bar x^k)\\
 &\geq f(\bar x^k_1)M_{\la_k}(\bar x^k)w_{\la_k}(\bar x^k )-\frac{c_2 m_k^2 }{d_k^{2s^2}}- \frac{c_3 m_k^2 }{d_k^{2s}}\\
  &\geq f(\lambda_0) M_{\la_k}(\bar x^k)w_{\la_k}(\bar x^k )-\frac{c_2 m_k^2 }{d_k^{2s^2}}- \frac{c_3 m_k^2 }{d_k^{2s}},
  \end{aligned}
  \end{equation}
  where $\bar x^k_1$ is the first component of $\bar x^k $ and  we have used the fact that $$|(-\lap)^s \varsigma_k(\bar x^k)|\leq \frac{c_2   }{d_k^{2s^2}} \,\, \mbox{and}\,\, |\lap \varsigma_k(\bar x^k)|\leq  \frac{c_3  }{d_k^{2s}}$$
  and 
  $$
  f(\bar x_1^k) \leq   f(\lambda_k) \leq f(\lambda_0).
  $$
  Putting together \eqref{finalde} and \eqref{ep1}, we conclude that 
  \begin{equation}\label{23e26}
  c_1(1-m_k ) \leq |f(\lambda_0)|M_{\la_k}(\bar x^k) d_k^{2s} +c_2m_k d_k^{2s(1-s)} +c_3m_k.
  \end{equation}
 If \eqref{22e26} is not true, then $d_k\rightarrow 0$ as $k\rightarrow +\infty.$ Noting that $m_k\rightarrow 0$ as $k\rightarrow +\infty,$ therefore, if we let $k \to \infty$ in \eqref{23e26}, we  derive a contradiction. Hence, \eqref{22e26} holds.

 Suppose, on the contrary, that \eqref{22e26} fails. Then we must have $d_k \to 0$ as $k \to +\infty$. Observing that $m_k \to 0$ as well, we may let $k \to \infty$ in \eqref{23e26} to arrive at a contradiction. Consequently, \eqref{22e26} is verified.
\vspace{0.2cm}

  Next, we continue to show $\lambda_0=+\infty$. Accordingly, we  perturb $w_{\la_k}$ near $x^k$ again as follows
\begin{equation}\label{02z}
\tilde V_k(x)=w_{\la_k}(x)-m_k^2 \tilde \varsigma_k(x),~x\in \mathbb R^n,
\end{equation}
where $$\tilde \varsigma_k(x)=\tilde \varsigma\left(\frac{x-x^k}{d _k}\right)$$ with $d_k=\frac{1}{2}dist\{x^k,T_{\la_k}\}\geq \bar C>0$ and $\tilde \varsigma \in C^{\infty}_0(\mathbb R^n)$ is a cut-off smooth function satisfying
 $$\left\{\begin{array}{ll}
0\leq \tilde \varsigma \leq 1,& \mbox{in}~\mathbb R^n,\\
 \tilde \varsigma=1,&  \mbox{in}~B_{\frac{1}{2}}(0),\\
 \tilde \varsigma=0, & \mbox{in}~\mathbb R^n \backslash B_1(0) .
\end{array}
\right.$$
It then follows from \eqref{e12}, \eqref{112z} and \eqref{02z} that 
\begin{equation*}
\begin{cases}
\tilde V_k(x^k)= w_{\la_k} (x^k) -m_k^2\tilde\varsigma_k(x^k)=-m_k ,&\\
\tilde V_k(x)=w_{\la_k} (x ) -m_k^2\tilde\varsigma_k(x )\geq  -m_k ,&~\mbox{in}~ \mathbb R^n \setminus B_{d_k }(x^k),\\
\tilde V_k(x )=  -m_k^2\tilde\varsigma_k(x )\geq  -m_k ,&~\mbox{on}~ T_{\la_k}.
\end{cases}
\end{equation*}
This implies that $\tilde{V}_k$ must attain its minima, which is at most $-m_k$, within $\overline{B_{d_k}(x^k)} \subset \Sigma_{\lambda_k}$ (see Figure~2 below), more precisely, 
\begin{eqnarray}\label{0z1}
\exists~ \{ \tilde x^k \} \subset \overline{B_{d _k}(x^k) } ~s.t. ~-m_k-m^2_k\leq \tilde V_k(\tilde x^k)=\underset{\Sigma_{\la_k}  }{\inf}~ \tilde V_k(x )\leq -m_k.
\end{eqnarray}
 \begin{center}
\begin{tikzpicture}[scale=1.5][node distance = 5mm]
\draw[ blue] (1 ,0.75 ) circle (1);
\path (1 ,0.75) [very thick,fill=blue]  circle(1pt) node at (1.02 ,1.02) [blue]  {{$ x^k  $}};
\path (1.45 ,0.45) [very thick,fill=red]  circle(1pt) node at (1.51 ,0.68) [red] {{$\tilde x^k  $}};
\path (0.8 ,2.1) node at (1.26,2 ) [ font=\fontsize{8}{8}\selectfont] [blue]{$  B_{d_k }(x^k )$};
\draw [->,  semithick] (-0.7 ,-0.65) -- (3.3,-0.65) node[right] [ font=\fontsize{10}{10}\selectfont] {  {$x_1$}};
\draw [semithick] (2.65,-1.1) -- (2.65,2.4) node[above] [ font=\fontsize{10}{10}\selectfont] {{$ T_{\lambda_k}$}};
\path (0.26 , 0.3) node at (0.27, 2.2) [ font=\fontsize{10}{10}\selectfont] {$ \Sigma_{\lambda_k}$} ;

\node [below=2mm, align=flush center,text width=88mm] at (1.2,-1.2)
        { $Figure\, 2$.  The domain  $ B_{d_k }(x^k) $     \\ \footnotesize
           }  ;
\end{tikzpicture}
\end{center}
 Hence, we obtain
\begin{eqnarray}\label{00ep1}
- m_k\leq  {w}_{\lambda_k}(\tilde x^k )\leq -m_k+m^2_k<0.
\end{eqnarray}

From \eqref{0z1}, using the same calculations as in \eqref{e25} and \eqref{e26}, but with $V_k$ and $\bar{x}^k$ replaced by $\tilde{V}_k$ and $\tilde{x}^k$, respectively, we obtain
\begin{eqnarray}\label{0e25}
-\lap \tilde V_k(\tilde x^k ) \leq 0,
\end{eqnarray}
and
\begin{equation}\label{0e26}
(-\lap)^s \tilde V_k(\tilde x^k )
\leq  \frac{c_1(-m_k+m^2_k)}{d_k^{2s}},
\end{equation}
Therefore, combining \eqref{02z},  \eqref{0e25}, \eqref{0e26} with \eqref{22e26}, we obtain
\begin{equation}\label{eqtper}
\begin{aligned}
 (-\lap)^sw_{\la_k}(\tilde x^k )-\lap  w_{\la_k}(\tilde x^k )
&=[(-\lap)^s-\lap]\tilde V_k(\tilde x^k)+m^2_k[(-\lap)^s-\lap]\tilde\varsigma_k(\tilde x^k)\\
&\leq \frac{c_1(-m_k+m^2_k)}{d_k^{2s}}+\frac{c_2  m_k^2 }{d_k^{2s }}+\frac{c_3 m_k^2 }{d_k^{2}}\\
&\leq \frac{c (-m_k+m^2_k)}{d_k^{2s}},
\end{aligned}
\end{equation}
where we have used the fact that $$|(-\lap)^s \tilde\varsigma_k(\tilde x^k)|\leq \frac{c_2   }{d_k^{2s }}\,\, \mbox{ and } \,\, |\lap \tilde\varsigma_k(\tilde x^k)|\leq  \frac{c_3  }{d_k^{2 }}.$$
Combining \eqref{eqtper} with \eqref{eq2w1} yields
\begin{equation}\label{ep2}
\frac{c(-m_k+m^2_k)}{d_k^{2s}}\geq f( \tilde x^k_1) M_{\la_k} (\tilde x^k ) w_{\la_k}(\tilde x^k ).
\end{equation}
If $f(\tilde{x}_1^k) \leq 0$, this leads to a contradiction immediately, and the proof is complete. Hence, we may assume that $f(\tilde{x}_1^k) > 0$. By \eqref{00ep1}, we arrive at
\begin{eqnarray}\label{ep3}
  c(1-m_k ) \leq f(\tilde x_1^k)M_{\lambda_k} (\tilde x^k)d_k^{2s}\leq M_{\lambda_k} (\tilde x^k) f(\tilde x_1^k)|\tilde x_1^k|^{2s}  .
\end{eqnarray}
 Therefore,   by assumption \eqref{p100} and the boundedness of $M_{\lambda_k} (\tilde x^k ) $, we obtain
 $$\{\tilde x^k_1\}~\mbox{is bounded away from}~-\infty$$
 which together with $f\in C(\mathbb R^n)$ and $ \tilde x^k\in  \overline{B_{d _k}(x^k) }$  with $d_k=\frac{1}{2} dist\{\tilde x^k,T_{\la_k}\}$ gives rise to
 $$f(\tilde x_1^k)\leq c,~ d_k^{2s}\leq c,$$
for sufficiently large $k$.

On the other hand,  by \eqref{ep3}, we conclude that
\begin{eqnarray}\label{ep6a}
f(\tilde x_1^k),~ M_{\lambda_k}(\tilde x^k )\geq C>0,
\end{eqnarray}
for sufficiently large $k$. Then by the definition of $M_{\lambda_k}$, the assumption \eqref{www} on $g$ and the positivity of $u$, we obtain 
\begin{eqnarray}\label{ep6}
 u(\tilde x^k )\geq C>0~ \mbox{and}~g(u(\tilde x^k ))\geq C>0.
\end{eqnarray}
Since $w_{\la_k}(\tilde x^k)\rightarrow 0$ as $k\rightarrow\infty$, by \eqref{ep6}, we have
\begin{eqnarray}\label{ep7}
 g( u_{\la_k}(\tilde x^k ))\geq C>0,
\end{eqnarray}
for sufficiently large $k$.

Next, starting from the original equation \eqref{eq2w1}, we revise \eqref{ep2} as follows:
\begin{equation}\label{ep8}
   \frac{c(-m_k + m_k^2)}{d_k^{2s}} \geq \left( f((\tilde{x}_1^k)^{\lambda_k}) - f(\tilde{x}_1^k) \right) g(u_{\lambda_k}(\tilde{x}^k, \tilde{t}_k)) + f(\tilde{x}_1^k) M_{\lambda_k}(\tilde{x}^k) w_{\lambda_k}(\tilde{x}^k).
\end{equation}
Noting that $|{\tilde{x}_1^k} - \lambda_k| \sim d_k$, and using the continuity and monotonicity of $f$ along with \eqref{ep6a}, we obtain
\begin{eqnarray}\label{ep9}
  f((\tilde x_1^k)^{\la_k})-f (\tilde x_1^k) \geq C>0.
\end{eqnarray}
Combining \eqref{ep1}, \eqref{ep6a}, \eqref{ep7}, \eqref{ep8} and \eqref{ep9},  we arrive at a contradiction for sufficiently large $k$.

Therefore,  $\la_0=+\infty$.

\textit{Step 3}. Based on the already established result
\[
w_\lambda(x) \geq 0, \quad \forall\, x \in \Sigma_\lambda,\ \forall\, \lambda \in \mathbb{R},
\]
we will now further prove that
\begin{equation}\label{ep9ww}
w_\lambda(x) > 0, \quad \forall\, x \in \Sigma_\lambda,\ \forall\, \lambda \in \mathbb{R}.
\end{equation}

Otherwise, for some fixed $\lambda$, there exists $\bar{x} \in \Sigma_\lambda$ such that
\[
w_\lambda(\bar{x}) = \min_{x \in \Sigma_\lambda} w_\lambda(x) = 0.
\]
At this minimum point, we have
\[
(-\Delta)^s w_\lambda(\bar{x}) - \Delta w_\lambda(\bar{x}) \leq 0.
\]
On the other hand, from equation \eqref{eq2w1}, we derive
\begin{equation*}
\begin{aligned}
(-\Delta)^s w_\lambda(\bar{x}) - \Delta w_\lambda(\bar{x}) 
&= \big(f(\bar{x}_1^\lambda) - f(\bar{x}_1)\big) g(u_\lambda(\bar{x})) 
+ f(\bar{x}_1) \big(g(u_\lambda(\bar{x})) - g(u(\bar{x}))\big) \\
&= \big(f(\bar{x}_1^\lambda) - f(\bar{x}_1)\big) g(u_\lambda(\bar{x})) \\
&> 0,
\end{aligned}
\end{equation*}
where the last inequality follows from the positivity of $u$ and $g$, and the strict monotonicity of $f$ with respect to $x_1$. This leads to a contradiction. Hence, \eqref{ep9ww} is established.

We have thus shown that every positive solution must be strictly increasing in the $x_1$ direction. This completes the proof of Theorem~\ref{mthm12}.
\end{proof}

Based on the monotonicity results, we will prove the nonexistence of positive bounded solutions for the mixed local and  nonlocal elliptic equation  by a simple principle eigenvalue argument.

\begin{proof}[Proof of Theorem \ref{thmf10}]
Since condition \eqref{e13} holds and $f(x_1)$ is strictly increasing in $x_1$, we may assume that there exists a sufficiently large constant $R$ such that $f(x_1) > 0$ for all $x_1 \in (R - 2, \infty)$. Let $B_1(x_0)$ denote the unit ball centered at $x_0 = (R, 0, \ldots, 0)$, and let $\lambda_1$ be the first (nonzero) eigenvalue of the problem
\begin{equation*}
\left\{
\begin{array}{ll}
(-\Delta)^s \varphi(x) - \Delta \varphi(x) = \lambda_1 \varphi(x), & x \in B_1(x_0), \\
\varphi(x) > 0, & x \in B_1(x_0), \\
\varphi(x) = 0, & x \in B_1^c(x_0).
\end{array}
\right.
\end{equation*}
According to \cite{SVWZ}, the corresponding eigenfunction $\varphi(x)$ satisfies $\varphi \in C^{2,\kappa}_{\mathrm{loc}}(B_1(x_0)) \cap C^{1,\kappa}(\overline{B_1(x_0)})$ for any $\kappa \in (0,1)$.

Let $\xi_u = \min_{B_1(0)} u$. Since $u$ is positive, we have $\xi_u > 0$ and
\[
m_0 := \frac{g(\xi_u)}{\sup_{\mathbb{R}^n} u} > 0.
\]
Using the fact that $u(x)$ is monotone increasing in the $x_1$ direction, it follows that
\[
(-\Delta)^s u(x) - \Delta u(x) \geq f(R-1)\, m_0\, u(x), \quad x \in B_1(x_0).
\]
By taking $R$ sufficiently large, and noting that $f(x_1) \to +\infty$ as $x_1 \to +\infty$, we can ensure that
\[
(-\Delta)^s u(x) - \Delta u(x) \geq \lambda_1 u(x), \quad x \in B_1(x_0).
\]
Define 
\[
A_1 = \max_{x \in B_1(x_0)} \frac{\varphi(x)}{u(x)}, \quad \text{and} \quad D_1(x) = A_1 u(x).
\]
Then \( D_1(x) \) satisfies
\begin{equation*}
\left\{
\begin{array}{ll}
(-\Delta)^s (D_1(x) - \varphi(x)) - \Delta (D_1(x) - \varphi(x)) \geq 0, & x \in B_1(x_0), \\
D_1(x) - \varphi(x) > 0, & x \in \mathbb{R}^n \backslash B_1(x_0).
\end{array}
\right.
\end{equation*}
Applying the maximum principle for mixed local and nonlocal equations (Lemma~\ref{lemf4}), we obtain
\begin{equation}\label{eq:f53}
D_1(x) > \varphi(x), \quad x \in B_1(x_0).
\end{equation}

On the other hand, by the definition of $A_1$, there exists a point $\bar{x} \in B_1(x_0)$ such that
\[
D_1(\bar{x}) = \frac{\varphi(\bar{x})}{u(\bar{x})} u(\bar{x}) = \varphi(\bar{x}),
\]
which contradicts \eqref{eq:f53}.

Therefore, equation \eqref{v1} admits no bounded positive solution. This completes the proof of Theorem~\ref{thmf10}.
\end{proof}

\section{Liouville theorem for  mixed local and nonlocal indefinite parabolic equations}

In this section, we present the proofs of Theorems~\ref{mthm1} and~\ref{mthm2}. For Theorem~\ref{mthm1}, we focus on the distinct computational aspects associated with the operator $\partial_t^\alpha$ in the cases $\alpha = 1$ and $0 < \alpha < 1$. In the proof of Theorem~\ref{mthm2}, we treat the cases $\alpha = 1$ and $0 < \alpha < 1$ separately.

 Let $u(x, t)$ be a positive solution of equation \eqref{eq1}. To compare the values of $u(x,t)$ and its reflection $u_\lambda(x, t) := u(x^\lambda, t)$, we define
\[
w_\lambda(x, t) := u(x^\lambda, t) - u(x, t).
\]
Clearly, $w_\lambda(x,t)$ is antisymmetric in $x$ with respect to the hyperplane $T_\lambda$.

Then, for $(x, t) \in \mathbb{R}^n \times \mathbb{R}$, due to the strict monotonicity of $f$ with respect to $x_1$, we have
\begin{eqnarray}\label{eq3}
&&  {\partial^\alpha_ t}w_\lambda(x,t) +(-\Delta)^s w_\lambda(x,t)-\lap w_\la(x,t)\nonumber\\
&=&f( x_1^\lambda )g(u_\lambda(x, t))-f(x_1)g(u(x, t))\nonumber\\
&=& (f(x_1^\lambda)-f(x_1))g(u_\lambda(x, t))+f(x_1)(g(u_\lambda(x, t))-g(u(x,t)))\nonumber\\
&\geq & f(x_1)M_\lambda(x,t) w_\lambda(x,t),
\end{eqnarray}
where 
$$M_\lambda(x, t)=\frac{g(u_\lambda(x,t))-f(u(x,t))}{u_\lambda(x,t)-u(x,t)}$$
is bounded and nonnegative.
Now we begin with the proof of Theorem \ref{mthm1}.
\vspace{0.2cm}

\begin{proof} [Proof of Theorem \ref{mthm1}.]
We will carry out the proof in three steps.

\emph{Step 1.} In the first step, we show that for $\lambda$ sufficiently negative, the following holds:
\begin{equation}\label{eq-1}
w_\lambda(x, t) \geq 0, \quad (x, t) \in \Sigma_\lambda \times \mathbb{R}.
\end{equation}

We proceed by contradiction and use a perturbation argument. Suppose \eqref{eq-1} does not hold. Since $u$ is bounded, there exists some $\lambda < -R$ for a constant $R > 0$ to be determined later, and a constant $B > 0$ such that
\begin{equation}\label{eq-2}
\inf_{(x, t) \in \Sigma_\lambda \times \mathbb{R}} w_\lambda(x, t) = -B < 0.
\end{equation}

If the minimum of $w_\lambda$ is attained at some point, then a contradiction can be derived directly at that point. However, since both $x$ and $t$ vary over unbounded domains, the minimum may not be achieved, and this case requires a more delicate treatment.

From \eqref{eq-2}, we deduce that there exists a sequence of approximate minima $\{(x^k, t_k)\} \subset \Sigma_\lambda \times \mathbb{R}$ and a sequence $\{\delta_k\} \searrow 0$ such that
\begin{equation}\label{eq-3}
w_\lambda(x^k, t_k) = -B + \delta_k < 0.
\end{equation}

To construct a sequence of functions that attain their minima, we perturb $w_\lambda$ near $(x^k, t_k)$ as follows:
\begin{eqnarray}\label{eq-4}
W_k(x , t ) =w_\la(x,t)-\delta_k\xi_k(x,t),~(x,t)\in \mathbb R^n\times \mathbb R,
\end{eqnarray}
where
\[
\xi_k(x,t) = \xi\left( \frac{x - x^k}{d_k}, \frac{t - t_k}{d_k^{\frac{2s}{\alpha}}} \right),
\]
with \( d_k = \frac{1}{2} \operatorname{dist}(x^k, T_\lambda) > 0 \), and \( \xi \in C_0^\infty(\mathbb{R}^n \times \mathbb{R}) \) is a smooth cutoff function satisfying
\[
\left\{
\begin{array}{ll}
0 \leq \xi \leq 1, & \text{in } \mathbb{R}^n \times \mathbb{R},\\[5pt]
\xi = 1, & \text{in } B_{\frac{1}{2}}(0) \times \left[-\tfrac{1}{2}, \tfrac{1}{2} \right],\\[5pt]
\xi = 0, & \text{in} \left(\mathbb{R}^n \times \mathbb{R}\right)\backslash \left(B_1(0) \times [-1, 1]\right).
\end{array}
\right.
\]

We now claim that the sequence \( \{x^k\} \) does not approach the hyperplane \( T_\lambda \), i.e., there exists a constant \( \tilde{c} > 0 \) such that
\begin{equation}\label{bbz}
d_k = \tfrac{1}{2} \operatorname{dist}(x^k, T_\lambda) \geq \tilde{c}, \quad \text{for all } k \in \mathbb{N}.
\end{equation}
If this is not true, then \( x^k \to T_\lambda \) and  for sufficiently large \( k \), we have
\[
W_k(x^k, t_k) = w_\lambda(x^k, t_k) - \delta_k \xi_k(x^k, t_k) = w_\lambda(x^k, t_k) - \delta_k \to 0.
\]
On the other hand, from \eqref{eq-3}, we know
\[
W_k(x^k, t_k) = w_\lambda(x^k, t_k) - \delta_k \xi_k(x^k, t_k) = -B < 0,
\]
which yields a contradiction, and hence \eqref{bbz} holds.

We now define the parabolic cylinder centered at \( (x^k, t_k) \) as
\[
E(x^k, t_k) := B_{d_k}(x^k) \times \left[ t_k - d_k^{\frac{2s}{\alpha}},\, t_k + d_k^{\frac{2s}{\alpha}} \right] \subset \Sigma_\lambda \times \mathbb{R},
\]
see Figure 3 below.

 \begin{center}
\begin{tikzpicture}[scale=1.8][node distance = 5mm]
\draw (-1,1.5) -- (-1,0) arc (180:360:1cm and 0.5cm) -- (1,1.5) ++ (-1,0) circle (1cm and 0.5cm);
\draw[densely dashed] (-1,0) arc (180:0:1cm and 0.5cm);
\path (1,1.5) [very thick,fill=black]  circle(1pt) node at (1.78,1.5) [ font=\fontsize{10}{10}\selectfont] {{$t_k+(d_k)^{\frac{2s}{\alpha}}$}};
\path (1,0) [very thick,fill=black]  circle(1pt) node at (1.76,-0.07) [ font=\fontsize{10}{10}\selectfont] {{$t_k-(d_k)^{\frac{2s}{\alpha}}$}};
\path (0,0.75) [very thick,fill=blue]  circle(1pt) node at (0.5,0.85) [ font=\fontsize{8}{8}\selectfont] [blue]{{$(x^k,t_k)$}};
\path (-0.6,0.9) [very thick,fill=red]  circle(1pt) node at (-0.6,0.7) [ font=\fontsize{8}{8}\selectfont] [red]{{$(\tilde x^k,\tilde t_k)$}};
\path (0,2.1) node at (0,2.2) [ font=\fontsize{8}{8}\selectfont] {$  E(x^k,t_k)$};
\draw [->,  semithick] (-2.2,-0.65) -- (3,-0.65) node[right] [ font=\fontsize{10}{10}\selectfont] {  {$x_1$}};
\draw [semithick] (2.75,-0.5,0.9) -- (2.95,0.6,-0.6  ) node[above] [ font=\fontsize{10}{10}\selectfont] {{$ T_\lambda$}};
\path (2.8 , 0.3) node at (2.6, 0.55) [ font=\fontsize{8}{8}\selectfont] {$ \Sigma_\lambda$} ;
\node [below=2mm, align=flush center,text width=88mm] at (0,-1.3)
        { $Figure ~ 3$.  The domain  $E(x^k,t_k)$  };
\end{tikzpicture}
\end{center}
Here, for the case of the mixed local and nonlocal parabolic equation with \(\alpha = 1\), the parameter \(\alpha\) in the expressions for \(\xi_k(x, t)\) and \(E(x^k, t_k)\) should also be taken as 1.

It follows from \eqref{eq-2}, \eqref{eq-3}, and \eqref{eq-4} that 
\[\left\{
\begin{array}{ll}
W_k(x^k, t_k) = -B, & \\
W_k(x, t) = -\delta_k \xi_k(x, t) > -B & \mbox{on}\,\, T_\lambda \times \mathbb{R},\\
W_k(x, t) = w_\lambda(x, t) \geq -B & \mbox{in}\,\, (\Sigma_\lambda \times \mathbb{R}) \setminus E(x^k, t_k).
\end{array}
\right.
\]
This implies that \(W_k\) must attain its minimum, which is at most \(-B\), at some point \((\tilde{x}^k, \tilde{t}_k) \in \overline{E(x^k, t_k)} \subset \Sigma_\lambda \times \mathbb{R}\), that is,
\begin{eqnarray}\label{23}
\exists~ \{(\tilde x^k,\tilde t_k)\} \subset \overline{E(x^k,t_k)}~s.t. ~-B-\delta_k\leq W_k(\tilde x^k,\tilde t_k)=\underset{\Sigma_\la \times \mathbb R}{\inf} W_k(x,t)\leq -B.
\end{eqnarray}
As a consequence of \eqref{eq-2} and \eqref{eq-4}, we have
\begin{eqnarray}\label{24}
\begin{cases}
-B\leq w_\la(\tilde x^k,\tilde t_k)\leq -B+\delta_k<0,&\\
\partial _tW_k(\tilde x^k,\tilde t_k)\leq 0, &\\
\partial^\alpha_tW_k(\tilde x^k,\tilde t_k)=C_\alpha \int_{-\infty}^{\tilde t_k} \frac{W_k(\tilde x^k,\tilde t_k)-W_k(\tilde x^k,\tau)}{(\tilde t_k-\tau)^{1+\alpha}}d\tau \leq 0, &\mbox{if}~0<\alpha<1.
\end{cases}\end{eqnarray}
From \eqref{23} and the definition of operator $(-\lap)^s$ and $-\lap$, and using the antisymmetry of $W_\la$ with respect to  $x$, we deduce
\begin{eqnarray}\label{25}
-\lap W_k(\tilde x^k,\tilde t_k) \leq 0,
\end{eqnarray}
and
\begin{equation}\label{26} \aligned
(-\lap)^s W_k(\tilde x^k,\tilde t_k)=&C_{n,s}P.V.\int_{\mathbb R^n} \frac{W_k(\tilde x^k,\tilde t_k)
-W_k(y,\tilde t_k)}{|\tilde x^k-y|^{n+2s}}dy\\
=&C_{n,s}P.V.\int_{\Sigma_\la} \frac{W_k(\tilde x^k,\tilde t_k)
-W_k(y,\tilde t_k)}{|\tilde x^k-y|^{n+2s}}dy\\
&+C_{n,s}\int_{\Sigma_\la} \frac{W_k(\tilde x^k,\tilde t_k)
-W_k(y^\la,\tilde t_k)}{|\tilde x^k-y^\la|^{n+2s}}dy\\
\leq& C_{n,s}\int_{\Sigma_\la} \frac{2W_k(\tilde x^k,\tilde t_k)-W_k(y,\tilde t_k)-W_k(y^\la,\tilde t_k)}{|\tilde x^k-y^\la|^{n+2s}}dy\\
=& C_{n,s}\int_{\Sigma_{\la}} \frac{2W_k(\tilde x^k,\tilde t_k)+\delta_k \xi_k(y)-\delta_k  \xi_k(y^{\lambda})}{|\tilde x^k-y^{\la}|^{n+2s}}dy\\
\leq & 2 C_{n,s}(W_k(\tilde x^k,\tilde t_k)+\delta_k)\int_{\Sigma_\la}\frac{1}{|\tilde x^k-y^\la|^{n+2s}}dy\\
\leq & 2 C_{n,s}(W_k(\tilde x^k,\tilde t_k)+\delta_k)\int_{\Sigma_\la}\frac{1}{|  x^k-y^\la|^{n+2s}}dy\\
\leq &\frac{c_1(-B+\delta_k)}{d_k^{2s}},
 \endaligned\end{equation}
 where the second  inequality from the bottom follows from the fact
 $$ |\tilde x^k-y^\la|\leq |\tilde x^k-x^k|+|x^k-y^\la|\leq \frac{3}{2}|  x^k-y^\la|.$$
Therefore, combining \eqref{eq-4}, \eqref{24}, \eqref{25}, \eqref{26} and \eqref{bbz},  we obtain, for sufficiently small $\delta_k$,
\begin{equation}\label{eqoes}
\begin{aligned}
&(\partial^\alpha_t+(-\lap)^s-\lap) w_\la(\tilde x^k,\tilde t_k)\\
=&(\partial^\alpha_t+(-\lap)^s-\lap) W_k(\tilde x^k,\tilde t_k)+\delta_k(\partial^\alpha_t+(-\lap)^s-\lap)\xi_k(\tilde x^k,\tilde t_k)\\
\leq & \frac{c_1(-B+\delta_k)}{d_k^{2s}}+\frac{c_2\delta_k}{d_k^{2s}}+\frac{c_3\delta_k}{d_k^2}\\
\leq & \frac{c (-B+\delta_k)}{d_k^{2s}},
\end{aligned}
\end{equation}
where we have used the fact that 
$$ \left|(\partial_t^\alpha+(-\lap)^s) \xi_k(\tilde x^k,\tilde t_k)\right|\leq \frac{c_2}{d_k^{2s}}\,\, \mbox{and}\,\, \left|\lap \xi_k(\tilde x^k,\tilde t_k)\right|\leq \frac{c_3}{d_k^2}.$$
Combining \eqref{eqoes} with \eqref{eq3}, one has
\begin{equation}\label{27}
\frac{c(-B+\delta_k)}{d_k^{2s}}\geq f( \tilde x^k_1) M_\lambda(\tilde x^k, \tilde t_k) w_\lambda(\tilde x^k,\tilde t_k).
\end{equation}
If $f( \tilde x^k_1)\leq0,$ this directly implies a contradiction and thus we complete the proof. If
$f( \tilde x^k_1)>0,$ by \eqref{eq-2}, \eqref{27} and $M_\lambda(\tilde x^k, \tilde t_k)$
 is bounded and nonnegative,  we arrive at
 $$\frac{c(-B+\delta_k)}{-B}\leq c_0f( \tilde x^k_1)d_k^{2s}\leq c_0f( \tilde x^k_1)|\tilde x^k_1|^{2s}.$$
 Hence for sufficiently large $k$, we obtain
$$ f( \tilde x^k_1)|\tilde x^k_1|^{2s}\geq \frac{c}{2c_0},$$
which is impossible, since $\lambda$ is sufficiently negative, $x^k_1 < \lambda$, and 
 $$\underset{x_1\rightarrow-\infty}{\limsup}f(x_1)|x_1|^{2s}\leq0.$$
Therefore, \eqref{eq-1} holds.

\textit{Step 2.}
We move the plane $T_\lambda$ to the right along the $x_1$-direction as long as the inequality \eqref{eq-1} holds, until it reaches its limiting position $T_{\lambda_0}$, where $\lambda_0$ is defined by
\[
\lambda_0 := \sup \left\{ \lambda \,\middle|\, w_\mu \geq 0 \ \text{in} \ \Sigma_\mu \times \mathbb{R} \ \text{for all} \ \mu \leq \lambda \right\}.
\]
We will prove
\begin{equation}\label{28}
\la_0=+\infty.
\end{equation}
If \eqref{28} does not hold, then $\lambda_0 < +\infty$. By the definition of $\lambda_0$, there exists a sequence $\{\lambda_k\} \searrow \lambda_0$ and a positive sequence $\{\gamma_k\}$ such that
\begin{equation*}
\underset{(x,t)\in \Sigma_{\la_k}\times \mathbb R}{\inf} w_{\la_k}(x,t):=-\gamma_k<0.
\end{equation*}

We first prove that 
\begin{equation}\label{30}
\gamma_k\rightarrow0 ~\mbox{as}~k\rightarrow \infty.
\end{equation}
Otherwise, there exists a subsequence, still denoted by $\{\gamma_k\}$, such that $\gamma_k<-m$ for some positive constant $m$. Hence, there exists a sequence $\{(y^k,s_k)\}\subset \Sigma_{\la_k}\times \mathbb R$ such that
\begin{equation}\label{31}
w_{\la_k}(y^k,s_k)\leq-m<0  .
\end{equation}

\textit{Case (i).} When $y^k \in \Sigma_{\lambda_k} \setminus \Sigma_{\lambda_0}$, together with the fact that $\lambda_k \to \lambda_0$ as $k \to \infty$, we obtain
\[
|y^k - (y^k)^{\lambda_k}| = 2|\lambda_k - y^k_1| \to 0 \quad \text{as} \quad k \to \infty.
\]
Then, by the uniform continuity of $u$, it follows that
\[
w_{\lambda_k}(y^k, s_k) = u((y^k)^{\lambda_k}, s_k) - u(y^k, s_k) \to 0 \quad \text{as} \quad k \to \infty.
\]

\textit{Case (ii).} When $y^k \in \Sigma_{\lambda_0}$, by the uniform continuity of $u$, the convergence $\lambda_k \to \lambda_0$, and the definition of $\lambda_0$, we deduce that
$$ \aligned w_{\la_k}(y^k,s_k)&= u((y^k)^{\la_k},s_k)-u((y^k)^{\la_0},s_k)+w_{\la_0}(y^k,s_k)\\
&\geq u((y^k)^{\la_k},s_k)-u((y^k)^{\la_0},s_k)\rightarrow  0,~\mbox{as}~ k\rightarrow \infty.
\endaligned$$
The conclusions from the above two cases contradict \eqref{31}. Therefore, \eqref{30} holds.

Secondly, from \eqref{30}, there exists a sequence $\{(x^k,t_k)\}\subset \Sigma_{\la_k}\times \mathbb R$ such that
$$
 w_{\lambda_k}(x^k, t_k)=-\gamma_k+\gamma_k^2<0.
$$
We will prove that  $x^k$ will not converge to $T_{\la_k}$,
namely,
\begin{equation}\label{e2622}
d_k=\frac{1}{2}dist\{x^k,T_{\la_k}\}\geq \hat C>0 .
\end{equation}
Otherwise, we have $d_k \to 0$. Then we construct a perturbation of $w_{\lambda_k}$ near the point $(x^k, t_k)$ as follows
$$\hat V_k(x,t)=w_{\la_k}(x,t)-\gamma_k^2\hat \xi_k(x,t),~(x,t)\in \mathbb R^n\times \mathbb R,$$
where $$\hat \xi_k =\hat \xi\left(\frac{x-x^k}{d^s_k},\frac{t-t_k}{{d_k}^{\frac{2s}{\alpha}}}\right)$$
with $d_k=\frac{1}{2}dist\{x^k,T_{\la_k}\} $  
and $\hat\xi\in C^{\infty}_0(\mathbb R^n\times \mathbb R)$ is a smooth cut off  functions satisfying
\begin{equation*}
\begin{cases}
 0\leq \hat\xi \leq 1, & \mbox{in}~ \mathbb R^n \times \mathbb R,\\
 \hat\xi=1, &  \mbox{in}~B_{\frac{1}{2}}(0)\times [-\frac{1}{2},\frac{1}{2}],\\
\hat\xi=0, & \mbox{in}~ (\mathbb R^n \times \mathbb R)\backslash (B_{1}(0)\times [-1,1]).
\end{cases}
\end{equation*}

Set
$$\hat E_k(x^k,t_k):=(B_{d^s_k}(x^k)\cap\Sigma_{\la_k}) \times [t_k-(d_k)^{\frac{2s}{\alpha}}, t_k+d_k^{\frac{2s}{\alpha}}]\subset \Sigma_{\la_k}\times \mathbb R.$$
As a result,  $\hat V_k$ must achieve its minimum which is at most $-\gamma_k$ in $\overline{\hat E_k(x^k,t_k)}\subset \Sigma_{\la_k}\times \mathbb R $ (see Figure 4),  that is
\begin{eqnarray*}
\exists~ \{(\hat x^k,\hat t_k)\} \subset \overline{\hat E_k(x^k,t_k)}~s.t. ~-\gamma_k-\gamma^2_k\leq \hat V_k(\hat x^k,\hat t_k)=\underset{\Sigma_{\la_k} \times \mathbb R}{\inf}\hat V_k(x,t)\leq -\gamma_k.
\end{eqnarray*}
 \begin{center}
\begin{tikzpicture}[scale=2][node distance = 5mm]
\draw (-1,1.5) -- (-1,0) arc (180:360:1cm and 0.5cm) -- (1,1.5) ++ (-1,0) circle (1cm and 0.5cm);
\draw[densely dashed] (-1,0) arc (180:0:1cm and 0.5cm);
\path (1,1.5) [very thick,fill=black]  circle(1pt) node at (1.6,1.7) [ font=\fontsize{10}{10}\selectfont] {{$t_k+(d_k)^{\frac{2s}{\alpha}}$}};
\path (1,0) [very thick,fill=black]  circle(1pt) node at (1.6, 0) [ font=\fontsize{10}{10}\selectfont] {{$t_k-(d_k)^{\frac{2s}{\alpha}}$}};
\path (0,0.75) [very thick,fill=blue]  circle(1pt) node at (0.4,0.85) [ font=\fontsize{8}{8}\selectfont][blue] {{$(x^k,t_k)$}};
\path (-0.7,0.5) [very thick,fill=red]  circle(1pt) node at (-0.7,0.7) [ font=\fontsize{8}{8}\selectfont][red] {{$(\hat x^k, \hat t_k)$}};
\path (0,2.1) node at (0,2.2) [ font=\fontsize{10}{10}\selectfont] {$ \hat E_k(x^k,t_k)$};
\draw [->,  semithick] (-2,-0.65) -- (2.5,-0.65) node[right] [ font=\fontsize{10}{10}\selectfont] {  {$x_1$}};
\draw [semithick] (0.75,-0.5,0.9) -- (0.95,0.6,-0.6 ) node[above] [ font=\fontsize{10}{10}\selectfont] {{$ T_{\lambda_k}$}};

\path (-1.2, 0.3) node at (-1.5, 0.5) [ font=\fontsize{10}{10}\selectfont] {$ \Sigma_{\lambda_k}$} ;
\node [below=2mm, align=flush center,text width=88mm] at (0,-1.2)
        { $Figure ~4$.  The domain $\hat E_k(x^k,t_k)$  };
\end{tikzpicture}
\end{center}
 It follows that 
\begin{eqnarray*}
-\gamma_k\leq  {w}_{\lambda_k}(\hat x^k, \hat t_k)\leq -\gamma_k+\gamma^2_k<0.
\end{eqnarray*}
Employing arguments parallel to those in \eqref{24}, \eqref{25}, and \eqref{26}, we conclude
$$\partial_t^\alpha \hat V_k(\hat x^k,\hat t_k)+ (-\lap)^s\hat V_k(\hat x^k,\hat t_k)-\lap\hat V_k(\hat x^k,\hat t_k)\leq \frac{\hat c(-\gamma_k+\gamma_k^2)}{d_k^{2s}},$$
which together with \eqref{eq3} gives rise to
$$\aligned\frac{\hat c(-\gamma_k+\gamma_k^2)}{d_k^{2s}}&\geq \partial_t^\alpha\hat V_k(\hat x^k,\hat t_k)+ (-\lap)^s\hat V_k(\hat x^k,\hat t_k)-\lap\hat V_k(\hat x^k,\hat t_k)\\
&=\left(\partial_t^\alpha + (-\lap)^s -\lap \right) w_{\la_k}(\hat x^k,\hat t_k)-\gamma_k^2\left(\partial_t^\alpha + (-\lap)^s -\lap \right) \hat \xi _k(\hat x^k,\hat t_k)\\
&\geq f(\hat x_1^k)M_{\la_k}(\hat x^k,\hat t_k)w_{\la_k}(\hat x^k,\hat t_k) -\frac{c_2\gamma_k^2}{d_k^{2s^2}}-\frac{c_3\gamma_k^2}{d_k^{2s}},
\endaligned $$
therefore,
\begin{equation*}
\begin{aligned}
\hat c(1-\gamma_k^2) \leq & f(\hat x_1^k)M_{\la_k}(\hat x^k,\hat t_k)d_k^{2s}+c_2\gamma_kd_k^{2s(1-s)}+c_3\gamma_k\\
\leq &  f(\lambda_0)M_{\la_k}(\hat x^k,\hat t_k)d_k^{2s}+c_2\gamma_kd_k^{2s(1-s)}+c_3\gamma_k.
\end{aligned}
\end{equation*}
This leads to a contradiction when $d_k \to 0$ as $k \to +\infty$. Therefore, \eqref{e2622} holds.

Next, we continue the proof of $\lambda_0 = +\infty$. To this end, we perturb $w_{\lambda_k}$ once again near the point $(x^k, t_k)$ as follows:
\[
V_k(x, t) = w_{\lambda_k}(x, t) - \gamma_k^2 \, \xi_k(x, t), \quad (x, t) \in \mathbb{R}^n \times \mathbb{R},
\]
where $\xi_k$ is defined as in Step 1, with 
\[
d_k = \frac{1}{2} \operatorname{dist}(x^k, T_{\lambda_k}) \geq \hat{C} > 0.
\]

Set
$$E_k(x^k,t_k):= B_{d _k}(x^k)  \times [t_k-(d_k)^{\frac{2s}{\alpha}}, t_k+d_k^{\frac{2s}{\alpha}}]\subset \Sigma_{\la_k}\times \mathbb R.$$
Then  $V_k$ must achieve its minimum which is at most $-\gamma_k$ within $\overline{E_k(x^k,t_k)}\subset \Sigma_{\la_k}\times \mathbb R $,  $i.e.,$
\begin{eqnarray*}
\exists~ \{(\tilde x^k,\tilde t_k)\} \subset \overline{E_k(x^k,t_k)}~s.t. ~-\gamma_k-\gamma^2_k\leq V_k(\tilde x^k,\tilde t_k)=\underset{\Sigma_{\la_k} \times \mathbb R}{\inf} V_k(x,t)\leq -\gamma_k,
\end{eqnarray*}
 which implies that 
\begin{eqnarray}\label{p1}
-\gamma_k\leq  {w}_{\lambda_k}(\tilde x^k, \tilde t_k)\leq -\gamma_k+\gamma^2_k<0.
\end{eqnarray}
Then, by applying an argument similar to that in Step 1, we obtain
\begin{equation}\label{p2}
\frac{c(-\gamma_k + \gamma_k^2)}{d_k^{2s}} \geq f(\tilde{x}_1^k) M_{\lambda_k}(\tilde{x}^k, \tilde{t}_k) w_{\lambda_k}(\tilde{x}^k, \tilde{t}_k).
\end{equation}
If $f(\tilde{x}_1^k) \leq 0$, this leads to a contradiction, and the proof is complete. Hence, we may assume that $f(\tilde{x}_1^k) > 0$. By \eqref{p1}, we then have
\begin{eqnarray}\label{p3}
  c(1-\gamma_k ) \leq f(\tilde x_1^k)M_{\la_k}(\tilde x^k,\tilde t_k)d_k^{2s}.
\end{eqnarray}
 Therefore,   by assumption \eqref{p100} in $(F1)$ and the boundedness of $M_{\la_k}(\tilde x^k,\tilde t_k)$, we obtain
 $$\{\tilde x^k_1\}~\mbox{is bounded away from}~-\infty$$
 which, together with  $f\in C(\mathbb R^n)$ with $(\tilde x^k,\tilde t_k)\in  \overline{E_k(x^k,t_k)}$  implies that
 $$f(\tilde x_1^k)\leq c,~ d_k^{2s}\leq c,$$
for sufficiently large $k$.

Furthermore, in view of \eqref{p3}, we conclude that
\begin{eqnarray}\label{p6a}
f(\tilde x_1^k),~ M_{\la_k}(\tilde x^k,\tilde t_k)\geq C>0,
\end{eqnarray}
for sufficiently large $k$. 
Then, using the inequalities
\[
u_{\lambda_k}(\tilde{x}^k, \tilde{t}_k) < \eta_{\lambda_k}(\tilde{x}^k, \tilde{t}_k) < u(\tilde{x}^k, \tilde{t}_k)
\]
and the positivity of $u$, we obtain
\begin{equation}\label{p6}
u(\tilde{x}^k, \tilde{t}_k) \geq C > 0 \quad \text{and} \quad g\bigl(u(\tilde{x}^k, \tilde{t}_k)\bigr) \geq C > 0.
\end{equation}
Since $w_{\lambda_k}(\tilde{x}^k, \tilde{t}_k) \to 0$ as $k \to \infty$, it follows from \eqref{p6} that
\begin{eqnarray}\label{p7}
 g( u_{\la_k}(\tilde x^k,\tilde t_k))\geq C>0,
\end{eqnarray}
for sufficiently large $k$.

Next, starting from the original equation \eqref{eq3}, we refine the estimate \eqref{p2} and rewrite it as
\begin{equation}\label{p8}
\frac{c(-\gamma_k + \gamma_k^2)}{d_k^{2s}} \geq \big[ f((\tilde{x}_1^k)^{\lambda_k}) - f(\tilde{x}_1^k) \big] u_{\lambda_k}^p(\tilde{x}^k, \tilde{t}_k) 
+ f(\tilde{x}_1^k) M_{\lambda_k}(\tilde{x}^k, \tilde{t}_k) w_{\lambda_k}(\tilde{x}^k, \tilde{t}_k).
\end{equation}

Observing that $|\tilde{x}_1^k - \lambda_k| \sim d_k$, and using the continuity and monotonicity of $f$, along with \eqref{p6a}, we obtain
\begin{equation}\label{p9}
f((\tilde{x}_1^k)^{\lambda_k}) - f(\tilde{x}_1^k) \geq C > 0.
\end{equation}

Combining \eqref{p1}, \eqref{p6a}, \eqref{p7}, \eqref{p8}, and \eqref{p9}, we arrive at a contradiction for sufficiently large $k$.

Therefore, \eqref{28} is established.

\textit{Step 3.} We have already established that
\[
w_\lambda(x,t) \geq 0,\quad \forall~(x,t)\in \Sigma_\lambda\times \mathbb{R},~\forall~\lambda\in \mathbb{R}.
\]
We now proceed to show that the inequality is in fact strict, that is 
\begin{equation}\label{p4}
w_\lambda(x,t) > 0,\quad \forall~(x,t)\in \Sigma_\lambda \times \mathbb{R},~\forall~\lambda\in \mathbb{R}.
\end{equation}

Assume by contradiction that for some fixed $\lambda$, there exists a point $(\bar{x}, \bar{t}) \in \Sigma_\lambda \times \mathbb{R}$ such that
\begin{equation}\label{p0}
w_\lambda(\bar{x}, \bar{t}) = \min_{\Sigma_\lambda \times \mathbb{R}} w_\lambda(x,t) = 0.
\end{equation}

On the one hand, since $(\bar{x}, \bar{t})$ is a global minimum point, we have
\[
\partial_t^\alpha w_\lambda(\bar{x}, \bar{t}) \leq 0,\quad (-\Delta)^s w_\lambda(\bar{x}, \bar{t}) \leq 0,\quad \text{and} \quad -\Delta w_\lambda(\bar{x}, \bar{t}) \leq 0.
\]
Thus,
\begin{equation}\label{p5}
\partial_t^\alpha w_\lambda(\bar{x}, \bar{t}) + (-\Delta)^s w_\lambda(\bar{x}, \bar{t}) - \Delta w_\lambda(\bar{x}, \bar{t}) \leq 0.
\end{equation}

On the other hand, by the second equation in \eqref{eq3}, we have
\begin{align*}
&\partial_t^\alpha w_\lambda(x,t) + (-\Delta)^s w_\lambda(x,t) - \Delta w_\lambda(x,t)\\
=& \big( f(x_1^\lambda) - f(x_1) \big) g(u_\lambda(x,t)) 
 + f(x_1) \big( g(u_\lambda(x,t)) - g(u(x,t)) \big).
\end{align*}
Evaluating this at $(\bar{x}, \bar{t})$ and noting that $w_\lambda(\bar{x}, \bar{t}) = 0$, i.e., $u_\lambda(\bar{x}, \bar{t}) = u(\bar{x}, \bar{t})$, we obtain
\begin{align}\label{p8b22}
\partial_t^\alpha w_\lambda(\bar{x}, \bar{t}) + (-\Delta)^s w_\lambda(\bar{x}, \bar{t}) - \Delta w_\lambda(\bar{x}, \bar{t})
= \big( f(\bar{x}_1^\lambda) - f(\bar{x}_1) \big) g(u_\lambda(\bar{x}, \bar{t})).
\end{align}

Since $f$ is strictly increasing in $x_1$, $\bar{x}_1^\lambda > \bar{x}_1$, and $u > 0$, and assuming $g$ satisfies condition $(F2)$ (i.e., $g(u) > 0$ for $u > 0$), the right-hand side of \eqref{p8b22} is strictly positive. However, this contradicts \eqref{p5}, which implies the left-hand side is nonpositive.

Therefore, the assumption in \eqref{p0} must be false, and \eqref{p4} holds. Consequently, $u(x, t)$ is strictly increasing in the $x_1$-direction.

This concludes the proof of Theorem~\ref{mthm1}.
\end{proof}
\vspace{0.4cm}

Now that we have shown that any positive bounded solution of \eqref{eq1} is strictly increasing in the 
$x_1$-direction, we proceed to derive a contradiction based on this monotonicity. This will ultimately lead to the nonexistence of positive solutions to \eqref{eq1}.

Before proving Theorem~\ref{mthm2}, we first present a key lemma.

\begin{lemma}\label{m2m}(Maximum principle for mixed fractional time operators)
Suppose that \( u(t) \in \mathcal{L}^-_{\alpha} \cap \mathcal{L}^-_{\beta} \) satisfies
\begin{equation}\label{qq1}
\left\{
\begin{aligned}
a\, \partial_t^\alpha u(t) + b\, \partial_t^\beta u(t) - C u(t) &\geq 0, && t \in [1, T], \\
u(t) &> 0, && t \in (-\infty, 1],
\end{aligned}
\right.
\end{equation}
where \( a \geq 0 \), \( b \geq 0 \), \( C > 0 \) are constants, \( 0 < \alpha, \beta < 1 \), and \( ab \neq 0 \). Then
\begin{equation}\label{aa}
u(t) > 0, \quad \forall\, t \in [1, T].
\end{equation}
\end{lemma}

\begin{proof}
Suppose, on the contrary, that \eqref{aa} is false. Then there exists a first time \( \hat{t} \in (1, T] \) such that
\[
u(\hat{t}) = \inf_{t \in (-\infty, \hat{t}]} u(t) = 0.
\]
Using the assumption \( u(t) > 0 \) for \( t \leq 1 \) and the definition of the Caputo fractional derivatives, we have
\[
\begin{aligned}
a\, \partial_t^\alpha u(\hat{t}) + b\, \partial_t^\beta u(\hat{t}) - C u(\hat{t})
&\leq a\, \partial_t^\alpha u(\hat{t}) + b\, \partial_t^\beta u(\hat{t}) \\
&= a C_\alpha \int_{-\infty}^{\hat{t}} \frac{u(\hat{t}) - u(\tau)}{(\hat{t} - \tau)^{1+\alpha}}\, d\tau 
+ b C_\beta \int_{-\infty}^{\hat{t}} \frac{u(\hat{t}) - u(\tau)}{(\hat{t} - \tau)^{1+\beta}}\, d\tau \\
&< 0,
\end{aligned}
\]
since \( u(\hat{t}) = 0 \) and \( u(\tau) > 0 \) for all \( \tau < \hat{t} \). This contradicts the first inequality in \eqref{qq1}.

Therefore, the conclusion \eqref{aa} must be true, and the proof is complete.
\end{proof}

Now we begin with the proof of Theorem \ref{mthm2}.

\begin{proof}[Proof of Theorem \ref{mthm2}.] 
Suppose, for the sake of contradiction, that \eqref{eq1} admits a global positive bounded solution \( u \).

Consider the first eigenvalue problem
\begin{equation*}
\begin{cases}
(-\Delta)^s \varphi(x) - \Delta \varphi(x) = \lambda_1 \varphi(x), & x \in B_1(R+2, 0'), \\
\varphi(x) = 0, & x \in \mathbb{R}^n \setminus B_1(R+2, 0'),
\end{cases}
\end{equation*}
where the large parameter \( R \geq 0 \) will be determined later.

To perform integration by parts, we mollify \( \varphi(x) \) to obtain \( \varphi_1(x) = \rho \ast \varphi(x) \in C_0^\infty(\mathbb{R}^n) \), where \( \rho \in C_0^\infty(B_1(0)) \) is the standard mollifier satisfying \( \int_{\mathbb{R}^n} \rho(x)\, dx = 1 \).

It is clear that the support of \( \varphi_1 \) is contained in \( B_2(R+2, 0') \). Moreover, by Lemma~\ref{mollification}, we have
\begin{equation}\label{Appen}
(-\Delta)^s \varphi_1(x) - \Delta \varphi_1(x) \leq \lambda_1 \varphi_1(x), \quad \text{for all } x \in \mathbb{R}^n.
\end{equation}

We may assume that 
\begin{equation}\label{po0}
\int_{\mathbb R^n} \varphi_1(x)dx=1.
\end{equation}
Denote
\begin{equation}\label{po01}
\psi_R(t):= \int_{\mathbb R^n} u(x,t) \varphi_1(x)dx=\int_{B_2(R +2,0')} u(x, t)\varphi_1 (x)dx.
\end{equation}

{\textit{Case 1.}} When $ \alpha=1,$ by  Remark \ref{mrem-inter}, equation \eqref{eq1}, conditions (F1) and (F2), $\operatorname{supp} \varphi_1\subset B_2(R +2,0')$, and \eqref{Appen}, we obtain
\begin{eqnarray}\label{econ}
\frac{d }{dt}\psi_R(t)&=& -\int_{\mathbb{R}^n}((-\Delta)^su(x,t)-\lap u(x,t))\varphi_1 (x)dx + \int_{\mathbb{R}^n}f( x_1)  g(u (x,t)) \varphi_1 (x)dx\nonumber\\
&=&-\int_{\mathbb{R}^n}u(x,t)((-\Delta)^s \varphi_1 (x)-\lap \varphi_1(x))dx + \int_{B_2(R +2,0')} f(x_1) g(u(x,t))\varphi_1 (x)dx\nonumber\\
&\geq &- \lambda_1\int_{\mathbb R^n}u(x,t)\varphi_1 (x)dx +   \int_{B_2(R +2,0')} f(x_1 ) g(u(x,t))\varphi_1 (x)dx\nonumber\\
&\geq & - \lambda_1 \psi_R(t) +   \frac{  f_R }{M}  \underset{B_2(0)\times[1,T]}{\inf} g(u(x,t))\int_{B_2(R +2,0')}  u(x,t)\varphi_1 (x)dx \nonumber\\
&=:&(-\la_1+\frac{f_R}{M}m_T)  \psi_R (t),
\end{eqnarray}
where $~M:=\underset{\mathbb R^n \times \mathbb R}{\sup}u,~m_T=\underset{B_2(0)\times[1,T]}{\inf} g(u(x,t)),~f_R \rightarrow +\infty $ as $R\rightarrow \infty$ and we have used the fact that the strictly increasing of $u$ with respect to $x_1$.
 Since \( u(x,t) \) is strictly increasing in the \( x_1 \)-direction by Theorem~\ref{mthm1}, it follows that for any fixed \( t \in \mathbb{R} \), the function \( \psi_R(t) \) is monotone increasing with respect to \( R \). Therefore,
\begin{eqnarray}\label{epsi0}
\psi_R(0)\geq 2C_0:=\psi_0(0).
\end{eqnarray}
If \( t \geq 0 \) satisfies \( \psi_R(t) \geq C_0 \), then, since 
\[
\lim_{x_1 \to +\infty} f(x_1) = +\infty,
\]
by choosing \( R \) sufficiently large, we have
$$-\la_1+\frac{f_R}{M}m_T \geq 1. $$
Then   \eqref{econ} implies that
$$
\frac{d }{dt}\psi_R(t) \geq  \psi_R(t) .
$$
Thus,  from \eqref{epsi0},  we deuce
$$
\psi_R(t) \geq 2C_0 e^t.
$$
Next we verify  the condition 
\begin{eqnarray}\label{condition}
\psi_R(t) \geq C_0, \,\, \forall ~t \geq 0
\end{eqnarray}
by  contradiction.
Suppose that \eqref{condition} does not hold, then there exists $t_0 >0$ such that
$$
\psi_R(t_0)=C_0 \,\, \mbox{ and }\,\, \psi_R(t)>C_0 \,\, \mbox{ in }\,\, [0, t_0).
$$
It follows that $\psi_R(t) \geq 2C_0 e^t \geq 2C_0$ in $[0, t_0),$  which contradicts $\psi_R(t_0)=C_0.$ Hence, \eqref{condition} must hold.

Therefore, $\psi_R(t)$ is monotone increasing with respect to $t,$ and
 $$
 \psi_R(t) \geq 2C_0 e^t,\,\, \forall~ t \geq 0.
 $$
Consequently,
$$
\psi_R(t) \to +\infty, \,\, \mbox{as} \,\, t\to +\infty,
$$
which contradicts the boundedness of $u(x, t).$

{\textit{Case 2.}} When $0<\alpha<1,$ we define $$ T=\left(\frac{M}{c_1}\right)^\theta,$$
where
\begin{equation}\label{0qo}
\theta>0, ~M:=\underset{\mathbb R^n \times \mathbb R}{\sup}u ~\mbox{ and }~ 2c_1:=\underset{B_2(R+2,0')\times [0,1]}{\inf} u.
\end{equation}

By equation \eqref{eq1}, conditions (F1) and (F2), $u(x,\cdot)\in C^1(\mathbb R),\, \operatorname{supp} \varphi_1\subset B_2(R +2,0')$, Remark \ref{mrem-inter} and \eqref{Appen}, similarly to \eqref{econ}, we conclude that for any $t\in [1,T],$
\begin{eqnarray}\label{con}
\partial_t^\alpha \psi_R(t)&=& \int_{\mathbb R^n} \partial_t^\alpha u(x,t) \varphi_1(x) dx\nonumber\\
&\geq & - \lambda_1 \psi_R(t) +   \frac{  f_R }{M}  \underset{B_2(0)\times[1,T]}{\inf} g(u(x,t))\int_{B_2(R +2,0')}  u(x,t)\varphi_1 (x)dx \nonumber\\
&=:&\left(-\la_1+\frac{f_R}{M}m_T\right)  \psi_R(t),
\end{eqnarray}
where 
\[
m_T = \inf_{B_2(0) \times [1, T]} g(u(x,t)), \quad f_R \to +\infty \quad \text{as } R \to \infty,
\]
and we have used the fact that \( u \) is strictly increasing with respect to \( x_1 \).

Next, we construct an increasing sub-solution for the linear operator \( \partial_t^\alpha - \hat{C}_\alpha \) for some constant \( \hat{C}_\alpha > 0 \).

Set
\[
w(t) = t^\theta,
\]
where \( 0 < \theta = \frac{1}{2k+1} < \alpha \) for some integer \( k \). Direct calculation yields
\begin{equation}\label{0p}
\begin{cases}
\partial_t^\alpha w(t) = \hat{C}_\alpha t^{\theta - \alpha} \leq \hat{C}_\alpha w(t), & t \in [1, +\infty), \\
w(t) \leq 0, & t \in (-\infty, 0], \\
w(t) \leq 1, & t \in [0,1].
\end{cases}
\end{equation}

Note that \(\lim_{x_1 \to +\infty} f(x_1) = +\infty\). Hence, we may choose \( R \) sufficiently large so that
\begin{equation}\label{0o}
-\lambda_1 + \frac{f_R}{M} m_T > \hat{C}_\alpha.
\end{equation}

By the monotonicity of \( u(x,t) \) with respect to \( x_1 \) established in Theorem~\ref{mthm1}, and combining \eqref{po0}, \eqref{po01}, and \eqref{0qo}, we conclude
\begin{equation}\label{qpo}
\psi_R(t) \geq 2 c_1, \quad t \in [0,1].
\end{equation}

Define
\[
v(t) = \psi_R(t) - c_1 w(t).
\]
Then, by \eqref{con}, \eqref{0p}, \eqref{0o}, and \eqref{qpo}, the function \( v(t) \) satisfies
\[
\begin{cases}
\partial_t^\alpha v(t) - \hat{C}_\alpha v(t) \geq 0, & t \in [1, T], \\
v(t) > 0, & t \in (-\infty, 1].
\end{cases}
\]

Applying the maximum principle (Lemma~\ref{m2m}) with \( a=1 \) and \( b=0 \), it follows that
\[
v(t) > 0, \quad \text{for all } t \in [1, T].
\]
That is,
\[
\psi_R(t) > c_1 t^\theta, \quad t \in [1, T].
\]

Evaluating at \( t = T \) and using \eqref{0qo}, we obtain
\[
\psi_R(T) > M,
\]
which contradicts 
\[
\psi_R(t) \leq M,
\]
where the last inequality follows from \eqref{po0} and \eqref{po01}.

This completes the proof of Theorem \ref{mthm2}.
\end{proof}

\bigbreak
\bigbreak
\bigbreak

\textbf{Acknowledgments}
P. Wang is supported by the   National Natural Science Foundation of China (No. 12101530), Sponsored by  the Program for Science \& Technology Innovation Talents in Universities of Henan Province (No. 26HASTIT040), Scientific and Technological Key Projects of Henan Province (No.232102310321) and Nanhu Scholars Program for Young Scholars of XYNU.

L. Wu is partially supported by National Natural Science Foundation of China (Grant No. 12401133) and the Guangdong Basic and Applied Basic Research Foundation (2025B151502069).

\vspace{2mm}

\textbf{Conflict of interest.} The authors do not have any possible conflict of interest.

\vspace{2mm}

\textbf{Data availability statement.}
 Data sharing not applicable to this article as no data sets were generated or analysed during the current study.

\end{document}